\documentclass[10pt]{amsart}
\usepackage{amsmath}
\usepackage{amsfonts}

\newtheorem{theorem}{Theorem}[section] 
\newtheorem{lemma}[theorem]{Lemma}     
\newtheorem{corollary}[theorem]{Corollary}
\newtheorem{proposition}[theorem]{Proposition}

\theoremstyle{definition}
\newtheorem{definition}[theorem]{Definition}
\newtheorem{example}[theorem]{Example}

\theoremstyle{remark}
\newtheorem{remark}[theorem]{Remark}
\numberwithin{equation}{section}

\newcommand{\C}{\mathbb{C}}
\newcommand{\R}{\mathbb{R}}
\newcommand{\N}{\mathbb{N}}
\newcommand{\de}{\partial}
\newcommand{\eps}{\varepsilon}
\newcommand{\phe}{\varphi}

\title[Toeplitz operators and Carleson measures]
 {Toeplitz operators and Carleson measures in strongly pseudoconvex domains} 

\author[Marco Abate]{Marco Abate}
\address{Marco Abate\\ Dipartimento di Matematica\\ Universit\`a di Pisa\\ Largo Pontecorvo 5, 56127 Pisa\\ Italy.} \email{abate@dm.unipi.it}
\author[Jasmin Raissy]{Jasmin Raissy*}
\address{Jasmin Raissy\\ Dipartimento Di Matematica e Applicazioni\\ Universit\`a degli Studi di Milano Bi\-coc\-ca\\ Via R. Cozzi 53, 20125 Milano\\ Italy. } \email{jasmin.raissy@unimib.it}
\author[Alberto Saracco]{Alberto Saracco}
\address{Alberto Saracco\\ Dipartimento di Matematica\\ Universit\`a di Parma\\ Parco Area delle Scienze 53/A, 43124 Parma\\ Italy. } \email{alberto.saracco@unipr.it}

\thanks{$^{*}$Partially supported by FSE, Regione Lombardia.}

\begin{document}

\begin{abstract}
We study mapping properties of Toeplitz operators associated to a finite positive Borel measure on a bounded strongly pseudoconvex domain $D\subset\subset\C^n$. In particular, we give sharp conditions on the measure ensuring that the associated Toeplitz operator maps the Bergman space $A^p(D)$ into~$A^r(D)$ with~$r>p$, generalizing and making more precise results by \v Cu\v ckovi\'c and McNeal. To do so, we give a geometric characterization of Carleson measures
and of vanishing Carleson measures of weighted Bergman spaces in terms of the intrinsic Kobayashi geometry of the domain, generalizing to this setting results obtained by Kaptano\u glu for the unit ball. 
\end{abstract}

\maketitle

\section{Introduction}
Let $X$ be a Hilbert algebra, $Y\subset X$ a Hilbert subspace, and $P\colon X\to Y$ the orthogonal projection. If $f\in X$ is given, the \emph{Toeplitz operator of symbol~$f$} is
the operator $T_f\colon X\to Y$ given by $T_f(g)=P(fg)$.

In complex analysis, the most important orthogonal projection is the \emph{Bergman projection}
of $L^2(D)$ onto $A^2(D)$, where $D\subset\subset\C^n$ is a  bounded domain, 
$A^p(D)=L^p(D)\cap\mathcal{O}(D)$ is the \emph{Bergman space} of $L^p$ holomorphic functions, and $\mathcal{O}(D)$ is the space of holomorphic functions on~$D$. The Bergman projection~$B$ is an integral operator of the form
\[
Bf(z)=\int_D K(z,w) f(w)\,d\nu(w)\;,
\]
where $K\colon D\times D\to\C$ is the \emph{Bergman kernel,} and $\nu$ denotes the Lebesgue measure. The Bergman projection has often been used to produce holomorphic functions having specific additional (e.g., growth) properties; to do so it has been necessary to study its mapping properties on more general Banach spaces, for instance $L^p$ spaces or H\"older spaces. 
This has been done for large classes of domains, e.g., strongly pseudoconvex domains and finite type domains (see \cite{PS}, \cite{AhS}, \cite{McS}, \cite{Bu}), where enough information on the 
boundary behavior of the Bergman kernel is known.  

It turns out that usually the Bergman projection maps continuously the given Banach spaces into themselves; and this is the best one can expect, because for each $p\ge 1$ in~$A^2(D)$ there are functions belonging to~$A^p(D)$ but to no~$A^q(D)$ for any $q>p$, and the Bergman projection is the identity on~$A^2(D)$. However, in many applications an operator creating holomorphic functions with better growth conditions might be useful. Thus \v Cu\v ckovi\'c and McNeal in \cite{CMc} suggested using special Toeplitz operators of the form
\[
T_{\delta^\eta}f(z)=B(\delta^\eta f)=\int_D K(z,w) f(w)\delta(w)^\eta\,d\nu(w)\;,
\]
where $\eta>0$ and $\delta(w)=d(w,\de D)$ is the euclidean distance from the boundary, and they were able to prove the following result:

\begin{theorem}[\textbf{\cite[Theorem 1.2]{CMc}}]
\label{th:CMc}
Let $D\subset\subset\C^n$ be a bounded strongly pseudoconvex domain, and let $\eta\ge 0$.
\begin{itemize}
\item[(a)] If $0\le\eta< n+1$, then:
\begin{itemize}
\item[(i)] if $1< p< \infty$ and $\frac{n+1}{n+1-\eta}<\frac{p}{p-1}$, then
$T_{\delta^\eta}\colon L^p(D)\to L^{p+G}(D)$ continuously, where $G=p^2\big/\left(\frac{n+1}{\eta}-p\right)$;
\item[(ii)] if $1< p< \infty$ and $\frac{n+1}{n+1-\eta}\ge\frac{p}{p-1}$, then $T_{\delta^\eta}\colon L^p(D)\to L^r(D)$ continuously for all $p\le r<\infty$.
\end{itemize}
\item[(b)] If $\eta\ge n+1$, then $T_{\delta^\eta}\colon L^1(D)\to L^\infty(D)$ continuously.
\end{itemize}
\end{theorem}

Related operators have also been studied (but only for $D=B^n$, the unit ball of~$\C^n$, and without discussing possible improvements in growth conditions) in \cite{FR}, \cite{Ka} and \cite{KZ}; in particular \cite{Ka} deals also with the problem of deciding when such operators are compact. 

\v Cu\v ckovi\'c and McNeal raised the question of whether the gain in the exponents obtained in Theorem~\ref{th:CMc} is optimal, and gave examples showing that this is the case when $n=1$; but they left open the problem for $n>1$. As a consequence of our work (see Theorem~\ref{th:corunoint} below), we shall be able to answer this question; more precisely, we shall be able to characterize completely, in a class of operators even larger than the one considered by \v Cu\v ckovi\'c and McNeal, the operators giving a specific gain in the exponents.

To express our results, let us first introduce the larger class of operators we are interested in. Given a finite positive Borel measure $\mu$ on~$D$, the \emph{Toeplitz operator} associated to~$\mu$ is given by
\[
T_\mu f(z)=\int_D K(z,w) f(w)\,d\mu(w)\;;
\]
clearly the Toeplitz operators $T_{\delta^\eta}$ considered by \v Cu\v ckovi\'c and McNeal are the Toeplitz operators associated to the measures $\delta^\eta\nu$. Similar operators were considered by Kaptano\u glu \cite{Ka} on the unit ball of~$\C^n$, and by
Schuster and Varolin \cite{ScV} in the setting of weighted Bargmann-Fock spaces on~$\C^n$.
They noticed relationships between the mapping properties of the Toeplitz operator~$T_\mu$ and Carleson properties of the measure~$\mu$ (still without considering possible gains in integrability). In this paper we shall precise, extend and considerably generalize these relationships in the setting of bounded strongly pseudoconvex domains; to do so we shall also prove a geometrical characterization of Carleson measures for weighted Bergman spaces.

Carleson measures have been introduced by Carleson \cite{C} in his celebrated solution of the corona problem in the unit disk of the complex plane, and, since then, have become an important tool in analysis, and an interesting object of study \emph{per se.} Let $A$ be a Banach space of holomorphic functions on a domain $D\subset\C^n$; given $p\ge 1$, a finite positive Borel measure $\mu$ on~$D$ is a \emph{Carleson measure} of~$A$ (for~$p$)
if there is a continuous inclusion $A\hookrightarrow L^p(\mu)$, that is there exists a constant $C>0$ such that
\[
\forall f\in A \qquad \int_D |f|^p\,d\mu\le C\|f\|_A^p\;;
\]
we shall furthermore say that $\mu$ is a \emph{vanishing Carleson measure} of~$A$ if the
inclusion $A\hookrightarrow L^p(\mu)$ is compact.

Carleson studied this property taking as Banach space $A$ the Hardy spaces $H^p(\Delta)$,
and proved that a finite positive Borel measure~$\mu$ is a Carleson measure of~$H^p(\Delta)$ for~$p$ if and only if there exists a constant $C>0$ such that
$\mu(S_{\theta_0,h})\le C h$ for all sets
\[
S_{\theta_0,h}=\{re^{i\theta}\in\Delta\mid 1-h\le r<1,\ |\theta-\theta_0|\le h\}
\]
(see also \cite{D}); in particular the set of Carleson measures of $H^p(\Delta)$ does not depend
on~$p$. 

In this paper we are however more interested in Carleson measures for Bergman spaces. 
In 1975, Hastings \cite{H} (see also Oleinik and Pavlov \cite{OP} and Oleinik \cite{O}) proved a similar characterization for the Carleson measures of the Bergman spaces~$A^p(\Delta)$, still expressed in terms of the sets~$S_{\theta,h}$. Later on, 
Cima and Wogen \cite{CW} have characterized Carleson measures for Bergman spaces in the unit ball $B^n\subset\C^n$, and Cima and Mercer \cite{CM} characterized Carleson measures of Bergman spaces in strongly pseudoconvex domains, showing in particular that the set of Carleson 
measures of $A^p(D)$ is independent of~$p\ge 1$, a typical feature of this subject. 

Cima and Mercer's characterization of Carleson measures of Bergman spaces is expressed using
suitable generalizations of the sets $S_{\theta,h}$; for our aims, it will be more useful a different characterization, expressed in terms of the intrinsic Kobayashi geometry of the domain. Given $z_0\in D$ and $0<r<1$, let $B_D(z_0,r)$ denote the ball of center~$z_0$ and radius~$\frac{1}{2}\log\frac{1+r}{1-r}$ for the Kobayashi distance~$k_D$ of~$D$ (that is,
of radius $r$ with respect to the pseudohyperbolic distance $\rho=\tanh(k_D)$; see Section~2 for the necessary definitions). Then it is possible to prove (see Luecking~\cite{Lu} for $D=\Delta$,
Duren and Weir~\cite{DW} and Kaptano\u glu~\cite{Ka} for $D=B^n$, and our previous paper~\cite{AS} for $D$ strongly pseudoconvex) that a finite positive measure~$\mu$ is a Carleson measure of~$A^p(D)$ for~$p$ if and only if for some (and hence all) $0<r<1$ there is a constant $C_r>0$
such that
\[
\mu\bigl(B_D(z_0,r)\bigr) \le C_r \nu\bigl(B_D(z_0,r)\bigr)
\]
for all $z_0\in D$. (The proof of this equivalence in~\cite{AS} relied on Cima and Mercer's characterization~\cite{CM}; in this paper we shall instead give a proof independent of~\cite{CM}, and of a more general result: see Theorem~\ref{carthetaCarluno}.)

Thus we have a geometrical characterization of Carleson measures of Bergman spaces, and it turns out that this \emph{geometrical} characterization is crucial for the study of the mapping properties of Toeplitz operators; but first (see also \cite{Ka}) it is necessary to widen the class of Carleson measures under consideration. Given $\theta>0$, we say that a finite positive Borel measure~$\mu$ is a
(geometric) \emph{$\theta$-Carleson measure} if for some (and hence all) $0<r<1$ there is a constant $C_r>0$
such that
\[
\mu\bigl(B_D(z_0,r)\bigr) \le C_r \nu\bigl(B_D(z_0,r)\bigr)^\theta
\]
for all $z_0\in D$; and we shall say
that $\mu$ is a (geometric) \emph{vanishing $\theta$-Carleson measure} if for some (and hence all) $0<r<1$ the quotient $\mu\bigl(B_D(z_0,r)\bigr)/\nu\bigl(B_D(z_0,r)\bigr)^\theta$ tends to~0
as $z_0\to\de D$.
In particular, a 1-Carleson measures is a usual Carleson measures of~$A^p(D)$, and we shall prove (see Theorem~\ref{th:ztre}) that $\theta$-Carleson measures are exactly the 
Carleson measures of suitably weighted Bergman spaces. Furthermore, it is easy to see that
when $D=B^n$ a $q$-Carleson measure in the sense of~\cite{Ka} is a $(1+\frac{q}{n+1})$-Carleson measure in our sense.

As a first example of the kind of results we shall be able to prove, we have the following theorem
(see Corollary~\ref{th:coruno}), implying in particular Theorem~\ref{th:CMc}, and answering
\v Cu\v ckovi\'c and McNeal's question about sharpness of the exponents:

\begin{theorem}
\label{th:corunoint}
Let $D\subset\subset\C^n$ be a bounded strongly pseudoconvex domain. 
Let $\mu$ be a finite positive Borel measure on $D$, and take $1< p< r<+\infty$. Then the following statements are equivalent:
\begin{itemize}
\item[(i)] $T_\mu\colon A^p(D)\to A^r(D)$ continuously (respectively, compactly);
\item[(ii)] $\mu$ is (respectively, vanishing) $\left(1+\frac{1}{p}-\frac{1}{r}\right)$-Carleson.
\end{itemize}
\end{theorem}

It is not difficult to prove (see Lemma~\ref{changeCarl}) that $\delta^\eta\nu$ is a $\left(1+\frac{\eta}{n+1}\right)$-Carleson measure; since
$\frac{1}{p}-\frac{1}{p+G}=\frac{\eta}{n+1}$ if and only if $G=p^2\bigg/\left(\frac{n+1}{\eta}-p\right)$,
it follows (see Corollary~\ref{corToeptre} for more details, explaining in particular why, in this case, mapping properties for~$A^p$ spaces imply mapping properties for~$L^p$ spaces) that Theorem~\ref{th:corunoint} 
does imply Theorem~\ref{th:CMc} and gives the best gain in integrability. 

The measures $\delta^\eta\nu$ are just one example of $\theta$-Carleson measures; a completely different kind of examples is provided by uniformly discrete sequences. A sequence $\{z_j\}\subset D$ is \emph{uniformly discrete} if there exists $\eps>0$ such that $k_D(z_j,z_k)\ge\eps$ for all $j\ne k$. Then (see Theorem~\ref{dtre}) it turns out that a sequence $\{z_j\}\subset D$ is a finite union of uniformly discrete sequences if and only if $\sum_j\delta(z_j)^{(n+1)\theta}\delta_{z_j}$ is a 
$\theta$-Carleson measure, where $\delta_{z_j}$ is the Dirac measure in~$z_j$, and as a consequence of Theorem~\ref{th:corunoint}
we obtain the following

\begin{corollary}
\label{th:coroundisc}
Let $D\subset\subset\C^n$ be a bounded strongly pseudoconvex domain, take $1<p<r<+\infty$
and let $\{z_j\}\subset D$ be a sequence of points in~$D$. Then the following statements are
equivalent:
\begin{itemize}
\item[(i)] $\{z_j\}$ is a finite union of uniformly discrete sequences;
\item[(ii)] for every $f\in A^p(D)$ the function
\[
g(z)=\sum_j K(z,z_j)f(z_j)\delta(z_j)^{(n+1)(1+\frac{1}{p}-\frac{1}{r})}
\]
belongs to~$A^r(D)$.
\end{itemize}
\end{corollary}

On the other hand, $\sum_j\delta(z_j)^{(n+1)\theta}\delta_{z_j}$ is a vanishing $\theta$-Carleson measure if and only if $\{z_j\}$ is a finite sequence; 
see Theorem~\ref{vandtre}.

To link $\theta$-Carleson measures and mapping properties 
of Toeplitz operators we use three main tools. The first one is a detailed study of the intrinsic (Kobayashi) geometry of strongly pseudoconvex domains, as performed in our previous paper~\cite{AS} and summarized in Section~2. The second one is a precise estimate (see Theorem~\ref{BKp}) of the integrability
properties of the Bergman kernel, done adapting techniques developed by McNeal and Stein~\cite{McS} and \v Cu\v ckovi\'c and McNeal~\cite{CMc}. The third one is a characterization of
(vanishing) $\theta$-Carleson measures involving both the Bergman kernel and an interpretation of $\theta$-Carleson measures as usual Carleson measures for weighted Bergman spaces.
Here, given $\beta\in\R$, the \emph{weighted Bergman space}
$A^p(D,\beta)$ is~$L^p(\delta^\beta \nu)\cap\mathcal{O}(D)$
endowed with the norm
\[
\|f\|_{p,\beta}=\left[\int_D |f(\zeta)|^p\delta^\beta(\zeta)\,d\nu(\zeta)\right]^{1/p}\;.
\]
To express our results, for each~$z_0\in D$ we shall denote by $k_{z_0}\colon D\to\C$ the 
\emph{normalized Bergman kernel} defined by 
\[
k_{z_0}(z)=\frac{K(z,z_0)}{\sqrt{K(z_0,z_0)}}=\frac{K(z,z_0)}{\|K(\cdot,z_0)\|_2}\;.
\]
The \emph{Berezin transform} of a finite positive Borel measure~$\mu$ on~$D$ is
the function $B\mu\colon D\to\R^+$ given by
\[
B\mu(z)=\int_D |k_z(w)|^2\,d\mu(w)\;.
\]
Then we shall be able to prove (see Theorems~\ref{carthetaCarluno} and~\ref{carthetaCarldue} for more details) the following characterization of $\theta$-Carleson measures, generalizing the characterization of Carleson measures given in~\cite{AS}:

\begin{theorem}
\label{th:ztre}
Let $D\subset\subset\C ^n$ be a bounded strongly
pseudoconvex domain. Then for each $1-\frac{1}{n+1}<\theta<2$ the following assertions are
equivalent:
\begin{itemize}
\item[(i)] $\mu$ is a Carleson measure of $A^p\bigl(D,(n+1)(\theta-1)\bigr)$, that is
$A^p\bigl(D,(n+1)(\theta-1)\bigr)\hookrightarrow L^p(\mu)$ continuously, for some (and hence all)
$p\in[1,+\infty)$;
\item[(ii)] $\mu$ is $\theta$-Carleson;
\item[(iii)] there exists $C>0$ such that $B\mu(z)\le C\delta(z)^{(n+1)(\theta-1)}$ for all $z\in D$.
\end{itemize}
\end{theorem}

Actually, it turns out that the implications (iii)$\Longleftrightarrow$(ii)$\Longrightarrow$(i) hold
for any $\theta>0$; see Remark~\ref{rem:uno}.

We also have a similar characterization (see Theorems~\ref{carvanthetaCarldue} and \ref{Carleson}) for vanishing $\theta$-Carleson measures, that, as far as we know, in the setting of strongly pseudoconvex domains is new for $\theta=1$ too:

\begin{theorem}
\label{th:zqua}
Let $D\subset\subset\C ^n$ be a bounded strongly
pseudoconvex domain. Then for each $1-\frac{1}{n+1}<\theta<2$ the following assertions are
equivalent:
\begin{itemize}
\item[(i)] $\mu$ is a vanishing Carleson measure of $A^p\bigl(D,(n+1)(\theta-1)\bigr)$, that is the inclusion
$A^p\bigl(D,(n+1)(\theta-1)\bigr)\hookrightarrow L^p(\mu)$ is compact, for some (and hence all)
$p\in[1,+\infty)$;
\item[(ii)] $\mu$ is vanishing $\theta$-Carleson;
\item[(iii)] we have $\delta(z)^{(n+1)(1-\theta)}B\mu(z)\to 0$ as $z\to\de D$.
\end{itemize}
\end{theorem}

Again, the implications (iii)$\Longrightarrow$(ii)$\Longrightarrow$(i) hold
for any $\theta>0$; see Remarks~\ref{rem:C} and~\ref{rem:Cb}.

The connection between Toeplitz operators and Berezin transform is given by the following useful formula (see Proposition~\ref{converse}):
\[
B\mu(z)=\int_D T_\mu k_z(w)\overline{k_z(w)}\,d\nu(w)\;.
\]
Using this, estimates on the normalized Bergman kernel and a few basic functional analysis arguments, we obtain information on the growth of~$B\mu$ from mapping properties of~$T_\mu$, and thus, via Theorems~\ref{th:ztre} and~\ref{th:zqua}, information on Carleson properties of the measure~$\mu$.
Conversely, the integral Minkowski inequality and the estimates on the Bergman kernel allow us to relate Carleson properties of~$\mu$---that is, by Theorems~\ref{th:ztre} and~\ref{th:zqua}, inclusions of suitable weighted Bergman spaces in $L^p(\mu)$--- with mapping properties of the associated Toeplitz operator. In this way, we obtain a large number of results, valid for
$p\in[1,+\infty]$, the case $p=1$ typically requiring a bit more care. The more general statements
(see Theorems~\ref{Toeptre}, \ref{Toepinftydue}, \ref{Toepqua}, \ref{Toepinfty}, \ref{Toepinvgen}, \ref{ToepinvLuno} and~\ref{Toepinvpinfty}) are a bit technical; but a few particular cases can
be stated more easily. One instance is Theorem~\ref{th:corunoint}; other examples are the following (see, respectively, Corollaries~\ref{th:cordue}, \ref{th:cortre}, \ref{th:cortreb} 
and~\ref{th:corqua}):

\begin{corollary}
\label{th:corduez}
Let $D\subset\subset\C^n$ be a bounded strongly pseudoconvex domain. 
Let $\mu$ be a finite positive Borel measure on $D$, and take $1\le p< r<p\left(1+\frac{1}{n}\right)$. Then the following statements are equivalent:
\begin{itemize}
\item[(i)] $T_\mu\colon A^p\bigl(D,(n+1)p(\frac{1}{r}-\frac{1}{p})\bigr)\to A^r(D)$ continuously (respectively, compactly);
\item[(ii)] $\mu$ is (respectively, vanishing) Carleson.
\end{itemize}
\end{corollary}

\begin{corollary}
\label{th:cortrez}
Let $D\subset\subset\C^n$ be a bounded strongly pseudoconvex domain. 
Let $\mu$ be a finite positive Borel measure on $D$, and take $1\le p<+\infty$. Then the following statements are equivalent:
\begin{itemize}
\item[(i)] $T_\mu\colon A^p\bigl(D,-(n+1)\eps\bigr)\to A^p(D)$ continuously (respectively, compactly) 
for all $\eps>0$;
\item[(ii)] $\mu$ is (respectively, vanishing) $\theta$-Carleson for all $\theta<1$.
\end{itemize}
\end{corollary}

\begin{corollary}
\label{th:cortrebz}
Let $D\subset\subset\C^n$ be a bounded strongly pseudoconvex domain. 
Let $\mu$ be a finite positive Borel measure on $D$, and take $1< r<+\infty$. Then the following statements are equivalent:
\begin{itemize}
\item[(i)] $T_\mu\colon A^1\bigl(D,-(n+1)\eps\bigr)\to A^r(D)$ continuously (respectively, compactly) 
for all $\eps>0$;
\item[(ii)] $\mu$ is (respectively, vanishing) $\theta$-Carleson for all $\theta<2-\frac{1}{r}$.
\end{itemize}
\end{corollary}

\begin{corollary}
\label{th:corquaz}
Let $D\subset\subset\C^n$ be a bounded strongly pseudoconvex domain. 
Let $\mu$ be a finite positive Borel measure on $D$, and take $1\le p<+\infty$. Then the following statements are equivalent:
\begin{itemize}
\item[(i)] $T_\mu\colon A^p\bigl(D,-(n+1)\eps\bigr)\to A^\infty(D)$ continuously (respectively, compactly) 
for all $\eps>0$;
\item[(ii)] $\mu$ is (respectively, vanishing) $\theta$-Carleson for all $\theta<1+\frac{1}{p}$.
\end{itemize}
\end{corollary}


We should mention that the techniques introduced here might work in other domains too (e.g., smoothly bounded convex domains of finite type, or finite type domains in~$\C^2$); but we restricted ourselves to strongly pseudoconvex domains to describe more clearly the main ideas.
Finally, we expressed our results in terms of the Lebesgue measure, and using the euclidean
distance from the boundary as weight, because this is the customary habit in this context;
however, it is possible to reformulate everything in completely intrinsic terms. Indeed,
let $\nu_D$ be the invariant Kobayashi measure (see, e.g., \cite{K}). Then \cite{M} implies that $A^p(\nu_D)=A^p\bigl(D,-(n+1)\bigr)$; furthermore it is well known (see, e.g., \cite{A}) that
 $\delta$ is bounded from above and below by constant multiples of~$\exp\bigl(2k_D(z_0,\cdot)\bigr)$ for any $z_0\in D$. Thus
all our statements can be reformulated in terms of completely intrinsic function spaces, 
using $\nu_D$ as reference measure and weights expressed in terms of the exponential of the Kobayashi distance
from a reference point.

\smallskip
The paper is structured as follows. In Section~2 we shall collect the preliminary results we need on the geometry of strongly pseudoconvex domains; in particular we shall prove (Theorem~\ref{BKp}) the integral estimates on the Bergman kernel mentioned before. In Section~3 we shall study $\theta$-Carleson measures, proving the characterizations described above; in Section~4 we shall analogously study vanishing $\theta$-Carleson measures, introducing the functional 
analysis results we shall need to deal with compactness properties of operators between
weighted Bergman spaces. Finally, in Section~5 we shall study the mapping properties
of Toeplitz operators, proving our main theorems.

\section{Preliminaries}
Let $D\subset\subset\C^n$ be a bounded strongly pseudoconvex domain in~$\C^n$. We shall
use the following notations:
\begin{itemize}
\item $\delta\colon D\to\R^+$ will denote the Euclidean distance from the boundary,
that is $\delta(z)=d(z,\de D)$;
\item given two non-negative functions $f$, ~$g\colon D\to\R^+$ we shall write $f\preceq g$
to say that there is $C>0$ such that $f(z)\le C g(z)$ for all $z\in D$. The constant $C$ is 
independent of~$z\in D$, but it might depend on other parameters ($r$, $\theta$, etc.);
\item given two strictly positive functions $f$, ~$g\colon D\to\R^+$ we shall write $f\approx g$
if $f\preceq g$ and $g\preceq f$, that is
if there is $C>0$ such that $C^{-1} g(z)\le f(z)\le C g(z)$ for all $z\in D$;
\item $\nu$ will be the Lebesgue measure;
\item $\mathcal{O}(D)$ will denote the space of holomorphic functions on~$D$, endowed with the topology of uniform convergence on compact subsets;
\item given $1\le p\le +\infty$, the \emph{Bergman space} $A^p(D)$ is the Banach space
$L^p(D)\cap\mathcal{O}(D)$, endowed with the $L^p$-norm;
\item more generally, given $\beta\in\R$ we introduce the \emph{weighted Bergman space}
\[
A^p(D,\beta)=L^p(\delta^\beta \nu)\cap\mathcal{O}(D)
\]
endowed with the norm
\[
\|f\|_{p,\beta}=\left[\int_D |f(\zeta)|^p\delta(\zeta)^\beta\,d\nu(\zeta)\right]^{1/p}
\]
if $1\le p<\infty$, and with the norm
\[
\|f\|_{\infty,\beta}=\| f\delta^\beta\|_\infty
\]
if $p=\infty$;
\item $K\colon D\times D\to\C$ will be the Bergman kernel of~$D$;
\item for each $z_0\in D$ we shall denote by $k_{z_0}\colon D\to\C$ the \emph{normalized
Bergman kernel} defined by 
\[
k_{z_0}(z)=\frac{K(z,z_0)}{\sqrt{K(z_0,z_0)}}=\frac{K(z,z_0)}{\|K(\cdot,z_0)\|_2}\;;
\]
\item given $r\in(0,1)$ and $z_0\in D$, we shall denote by $B_D(z_0,r)$ the Kobayashi ball
of center~$z_0$ and radius $\frac{1}{2}\log\frac{1+r}{1-r}$.
\end{itemize}

\noindent See, e.g., \cite{A,A1,JP,K} for definitions, basic properties and applications to geometric function theory of the Kobayashi distance; and \cite{Ho,Ho1,Kr,R} for definitions and basic properties of the Bergman kernel. 

Let us now recall a number of results proved in \cite{AS}. The first two give information about the shape of Kobayashi balls:

\begin{lemma}[\textbf{\cite[Lemma 2.1]{AS}}]
\label{sei} 
Let $D\subset\subset\C ^n$ be a bounded  strongly pseudoconvex domain, and $r\in(0,1)$. Then 
\[
\nu\bigl(B_D(\cdot,r)\bigr)\approx \delta^{n+1}\;,
\]
(where the constant depends on~$r$).
\end{lemma}

\begin{lemma}[\textbf{\cite[Lemma 2.2]{AS}}]
\label{sette} 
Let $D\subset\subset\C ^n$ be a bounded  strongly pseudoconvex domain. Then there is $C>0$ such that
\[
\frac{C}{1-r}\delta(z_0)\ge \delta(z) \ge \frac{1-r}{C}\delta(z_0)
\]
for all $r\in(0,1)$, $z_0\in D$ and $z\in B_D(z_0,r)$.
\end{lemma}

We shall also need the existence of suitable coverings by Kobayashi balls:

\begin{definition}
\label{def:lattice}
Let $D\subset\subset\C^n$ be a bounded domain, and $r>0$. An \emph{$r$-lattice} in~$D$
is a sequence $\{a_k\}\subset D$ such that $D=\bigcup_{k} B_D(a_k,r)$ and 
there exists $m>0$ such that any point in~$D$ belongs to at most $m$ balls of the
form~$B_D(a_k,R)$, where $R=\frac{1}{2}(1+r)$.
\end{definition}

The existence of $r$-lattices in bounded strongly pseudoconvex domains is ensured by
the following 

\begin{lemma}[\textbf{\cite[Lemma 2.5]{AS}}]
\label{uno}
Let $D\subset\subset\C ^n$ be a bounded strongly pseudoconvex domain. Then for every $r\in(0,1)$ there exists an $r$-lattice in~$D$, that is there exist
 $m\in{\bf N}$ and a sequence 
$\{a_k\}\subset D$ of points such that
$D=\bigcup_{k=0}^\infty B_D(a_k,r)$
and no point of $D$ belongs to more than $m$ of the balls $B_D(a_k,R)$,
where $R={\frac12}(1+r)$.
\end{lemma}

We shall use a submean estimate for nonnegative plurisubharmonic functions on Kobayashi balls:

\begin{lemma}[\textbf{(\cite[Corollary 2.8]{AS})}] 
 \label{due}
 Let $D\subset\subset\C ^n$ be a bounded strongly pseudoconvex domain. Given $r\in(0,1)$, set $R={\frac12}(1+r)\in(0,1)$. Then there exists a $K_r>0$ depending on~$r$ such that
\[
\forall{z_0\in D\;\forall z\in B_D(z_0,r)}\ \ \ \ \chi(z)\le 
{\frac{K_r}{\nu\left(B_D(z_0,r)\right)}}\int_{B_D(z_0,R)}\chi\,d\nu
\]
for every nonnegative plurisubharmonic function $\chi\colon D\to{\bf R}^+$.
\end{lemma}

We now collect a few facts on the (possibly weighted) $L^p$-norms of the Bergman kernel
and the normalized Bergman kernel. The first result is classical (see, e.g., \cite{Ho1}):

\begin{lemma}
\label{BKbasic}
Let $D\subset\subset\C ^n$ be a bounded strongly
pseudoconvex domain. Then
\[
\|K(\cdot,z_0)\|_2=\sqrt{K(z_0,z_0)}\approx \delta(z_0)^{-(n+1)/2}\qquad
\hbox{and}\qquad \|k_{z_0}\|_2\equiv 1
\]
for all $z_0\in D$.
\end{lemma}

The next result is the main result of this section, and contains the weighted $L^p$-estimates we shall need:

%

\begin{theorem}
\label{BKp}
Let $D\subset\subset\C ^n$ be a bounded strongly
pseudoconvex domain, and let $z_0\in D$ and $1\le p<\infty$. Then
\begin{equation}
\int_D |K(\zeta,z_0)|^p \delta(\zeta)^{\beta}\,d\nu(\zeta)\preceq 
\begin{cases}
\delta(z_0)^{\beta-(n+1)(p-1)}&\hbox{ for $-1<\beta<(n+1)(p-1)$;}\\
|\log\delta(z_0)|&\hbox{ for $\beta=(n+1)(p-1)$;}\\
1&\hbox{ for $\beta>(n+1)(p-1)$.}
\end{cases}
\label{eqBKp}
\end{equation}
In particular:
\begin{itemize}
\item[(i)] $\|K(\cdot,z_0)\|_{p,\beta}\preceq \delta(z_0)^{\frac{\beta}{p}-\frac{n+1}{q}}$
and $\|k_{z_0}\|_{p,\beta}\preceq \delta(z_0)^{\frac{n+1}{2}+\frac{\beta}{p}-\frac{n+1}{q}}$
when $-1<\beta<(n+1)(p-1)$, where $q>1$ is the conjugate exponent of~$p$ (and $\frac{n+1}{q}=0$ when $p=1$);
\item[(ii)] $\|K(\cdot,z_0)\|_{p,\beta}\preceq 1$
and $\|k_{z_0}\|_{p,\beta}\preceq \delta(z_0)^{\frac{n+1}{2}}$ when 
$\beta>(n+1)(p-1)$;
\item[(iii)] $\|K(\cdot,z_0)\|_{p,(n+1)(p-1)}\preceq \delta(z_0)^{-\eps}$
and $\|k_{z_0}\|_{p,(n+1)(p-1)}\preceq \delta(z_0)^{\frac{n+1}{2}-\eps}$ for all $\eps>0$.
\end{itemize}
Furthermore,
\begin{itemize}
\item[(iv)] $\|K(\cdot,z_0)\|_{\infty,\beta} \approx \delta(z_0)^{\beta-(n+1)}$ and
$\|k_{z_0}\|_{\infty,\beta} \approx \delta(z_0)^{\beta-(n+1)/2}$ for all $0\le\beta<n+1$;
and $\|K(\cdot,z_0)\|_{\infty,\beta} \approx 1$ and
$\|k_{z_0}\|_{\infty,\beta} \approx \delta(z_0)^{(n+1)/2}$ for all $\beta\ge n+1$.
\end{itemize}
\end{theorem}

\begin{proof}
We shall closely follow the argument of \cite[Proposition~3.4]{CMc}.

First of all, Kerzman~\cite{Ke} proved that the Bergman kernel of a bounded 
strongly pseudoconvex domain is smooth outside the boundary diagonal, that is
$K\in C^\infty(\overline{D}\times \overline{D}\setminus\Delta_\de)$, where $\Delta_\de=\{(x,x)\mid x\in\de D\}$. We also recall an (essentially sharp) estimate on the Bergman kernel which follows from Fefferman's expansion \cite{F}. Let $r\colon\C^n\to\R$ be a smooth defining function for~$D$,
that is $D=\{r<0\}$ and $dr\ne 0$ on $\de D$; since $D$ is strongly pseudoconvex, we can 
also assume that the Levi form of~$r$ is positive definite on~$\de D$. Notice that, being
$D$ bounded, we have $|r|\approx \delta$ on~$\overline{D}$. Then (see, e.g., \cite{Mc}) there is $C>0$ such that for each $x\in \de D$ we can find a neighborhood $U$ of~$x$ in~$\C^n$ and
local coordinates $\phe=(\phe_1,\ldots,\phe_n)\colon U\to\C^n$ so that 
\begin{equation}
|K(w,z)|\le C\left(|r(z)|+|r(w)|+|\phe_1(z)-\phe_1(w)|+\sum_{k=2}^n |\phe_k(z)-\phe_k(w)|^2\right)^{-(n+1)}
\label{eqKF}
\end{equation}
for all $z$, $w\in U\cap D$. 

Cover $\de D$ with a finite number $U_1,\ldots,U_m$ of such neighborhoods; we can also
assume that they are so small that the quantity in brackets in the right-hand side of (\ref{eqKF})
is always less than~1. Setting $U_0=D\setminus\bigcup_{j=1}^m U_j$, the smoothness of
the Bergman kernel outside the boundary diagonal implies that for any $p\ge 0$ and $\beta\in\R$ we have
\[
\int_{U_0} |K(\zeta,z_0)|^p\delta(\zeta)^\beta\,d\nu(\zeta)\preceq 1
\]
for all $z_0\in\overline{D}$; we must control the integral on $\bigcup_{j=1}^m U_j\cap D$.

We fix $1\le j\le m$, and we work in the local coordinated defined in~$U_j$ mentioned
above. Since $p\ge 1$ and the quantity in brackets is less than 1, we have
\[
\begin{aligned}
I_j&=\int_{U_j\cap D}|K(w,z)|^p\delta(w)^{\beta}\,d\nu(w)\cr
&\preceq \int_{\phe(U_j\cap D)}\left(|r(z)|+|r(w)|+|z_1-w_1|+\sum_{k=2}^n|z_k-w_k|^2\right)^{-p(n+1)}|r(w)|^{\beta}\,d\nu(w)\;.
\end{aligned}
\]
We change again coordinates, putting $\tilde w_k=w_k-z_k$ for $k=2,\ldots,n$, and $\tilde w_1=r(w)+i\textrm{Im}(w_1-z_1)$; we also put
$x=\textrm{Re}\,\tilde w_1$ and $y=\textrm{Im}\,\tilde w_1=\textrm{Im}\,(w_1-z_1)$. Then, since
$|z_1-w_1|\ge|y|$, we get
\[
I_j\preceq\int_W\left(|r(z)|+|x|+|y|+\sum_{k=2}^n |\tilde w_k|^2\right)^{-p(n+1)}|x|^\beta\,
d\tilde w_2\cdots d\tilde w_n\,dx\,dy\;,
\]
where $W=[0,d]\times\R\times\C^{n-1}$, and $d=\max_{w\in D}|r(w)|$.

Let us first perform the integration on $\tilde w_2$. Put
\[
\begin{aligned}
\Omega_1=\left\{\tilde w_2\in\C\,\Big|\, |\tilde w_2|^2>|r(z)|+|x|+|y|+\sum_{k=3}^n|\tilde w_k|^2\right\}\;,\\
\Omega_2=\left\{\tilde w_2\in\C\,\Big|\, |\tilde w_2|^2<|r(z)|+|x|+|y|+\sum_{k=3}^n|\tilde w_k|^2\right\}\;.
\end{aligned}
\]
Using polar coordinates in~$\Omega_1$ we get
\[
\begin{aligned}
\int_{\Omega_1}\left(|r(z)|+|x|+|y|+\sum_{k=2}^n |\tilde w_k|^2\right)^{-p(n+1)}\!\!\!\!\!|x|^\beta\,
d\tilde w_2&\le \int_{\Omega_1}(|\tilde w_2|^2)^{-p(n+1)}|x|^\beta\,d\tilde w_2\\
&\preceq\int_L^{+\infty} R^{-2p(n+1)}R|x|^\beta\,dR\\
&\preceq \left(|r(z)|+|x|+|y|+\sum_{k=3}^n|\tilde w_k|^2\right)^{-p(n+1)+1}\!\!\!\!|x|^\beta\;,
\end{aligned}
\]
where $L=\left(|r(z)|+|x|+|y|+\sum_{k=3}^n|\tilde w_k|^2\right)^{1/2}$. On~$\Omega_2$ we obtain the same upper bound just by a direct estimation:
\[
\begin{aligned}
\int_{\Omega_2}\left(|r(z)|+|x|+|y|+\sum_{k=2}^n |\tilde w_k|^2\right)^{-p(n+1)}&\!\!\!\!|x|^\beta\,
d\tilde w_2\cr
&\le \left(|r(z)|+|x|+|y|+\sum_{k=3}^n|\tilde w_k|^2\right)^{-p(n+1)}\!\!\!\!\!\!\!\!|x|^\beta
\textrm{Area}(\Omega_2)\\
&\preceq \left(|r(z)|+|x|+|y|+\sum_{k=3}^n|\tilde w_k|^2\right)^{-p(n+1)+1}|x|^\beta\;.
\end{aligned}
\]
We can do the same kind of computations on $\tilde w_3,\ldots,\tilde w_n$, reducing the negative
power by one at each step, until we obtain
\[
I_j\preceq \int_{[0,d]\times\R}\left(|r(z)|+|x|+|y|\right)^{-p(n+1)+n-1}|x|^\beta\,dx\,dy\;.
\]
Since $p\ge 1$, we have $-p(n+1)+n-1\le-2$; so we can perform once again the same
kind of integration on~$y$, obtaining
\[
I_j\preceq \int_0^d\bigl(|r(z)|+|x|\bigr)^{-p(n+1)+n}|x|^\beta\,dx
=|r(z)|^{\beta-(n+1)(p-1)}\int_0^{d/|r(z)|}\frac{t^\beta}{(1+t)^{p(n+1)-n}}\,dt\;.
\]
If $\beta>(n+1)(p-1)$ we have
\[
I_j\preceq |r(z)|^{\beta-(n+1)(p-1)}\int_0^{d/|r(z)|}t^{\beta-p(n+1)+n}\,dt=
\frac{d^{\beta-(n+1)(p-1)}}{\beta-(n+1)(p-1)}\preceq 1\;. 
\]
If instead $-1<\beta\le(n+1)(p-1)$ we can estimate as follows
\[
\begin{aligned}
|r(z)|^{\beta-(n+1)(p-1)}&\int_0^{d/|r(z)|}\frac{t^\beta}{(1+t)^{p(n+1)-n}}\,dt\cr
&\le|r(z)|^{\beta-(n+1)(p-1)}\left[\int_0^1\frac{t^\beta}{(1+t)^{p(n+1)-n}}\,dt
+\int_1^{d/|r(z)|}t^{\beta-p(n+1)+n}\,dt\right]\;.
\end{aligned}
\]
The first integral in square brackets is just a (finite because $\beta>-1$) constant. If $\beta=(n+1)(p-1)$ the second 
integral is $|\log|r(z)||+\log d$, and thus
\[
I_j\preceq \bigl|\log|r(z)|\bigr|\;.
\]
If $-1<\beta<(n+1)(p-1)$ the second integral is of the form $c_1-c_2|r(z)|^{-\beta+(n+1)(p-1)}$ for
suitable constants $c_1$, $c_2>0$, and thus we get
\[
I_j\preceq |r(z)|^{\beta-(n+1)(p-1)}\;.
\]
%
%
So we obtained the desired bound on $I_j$ as soon as $z\in U_j\cap D$. But if $z\notin U_j\cap D$
we have $|K(z,w)|\preceq 1$ for $w\in U_j\cap D$, and thus $I_j\preceq 1$ in this case.
Putting all together, we have proved (\ref{eqBKp}), and the rest of the statements (i)--(iii)
follows immediately recalling that $|k_{z_0}|\preceq \delta^{\frac{n+1}{2}}|K(\cdot,z_0)|$
and using Lemma~\ref{BKbasic}.

Finally, Kerzman's result~\cite{Ke} and \eqref{eqKF} yield
\[
|K(w,z_0)|\delta(w)^\beta\preceq\frac{|r(w)|^\beta}{\bigl(|r(z_0)|+|r(w)|\bigr)^{n+1}}\;.
\]
The supremum (in $w$) of the latter quantity is bounded by a constant
times $|r(z_0)|^{\beta-(n+1)}$ when $0\le\beta<n+1$, and is bounded by a constant
independent of~$z_0$ when $\beta\ge n+1$; recalling that (Lemma~\ref{BKbasic})  
$K(z_0,z_0)\approx\delta(z_0)^{-(n+1)}$ we obtain (iv).
\end{proof}

Another fact that shall be useful later on is:

\begin{lemma}
\label{BK}
Let $D\subset\subset\C^n$ be a bounded strongly pseudoconvex domain. Then
$\delta(z_0)^\beta k_{z_0}\to 0$ uniformly on compact subsets as $z_0\to\de D$
for all $\beta>-(n+1)/2$.
\end{lemma}

\begin{proof}
The already quoted result by
Kerzman \cite[Theorem 2]{Ke} of continuous extendibility of the Bergman kernel outside the boundary diagonal implies
that for every compact subset $D_0\subset\subset D$ we have
\[
\sup_{w\in D, w_0\in D_0}|K(w_0,w)| <+\infty\;.
\]
On the other hand, Lemma~\ref{BKbasic} yields
\[
|k_{z_0}(z)|\preceq\delta(z_0)^{(n+1)/2}|K(z,z_0)|\;.
\]
Therefore for every compact subset $D_0\subset\subset D$ we can find $C_{D_0}>0$ such that
\[
|k_{z_0}(z)|\le C_{D_0}\delta(z_0)^{(n+1)/2}
\]
for all $z\in D_0$ and $z_0\in D$, and we are done. 
\end{proof}

We also recall another result from~\cite{AS}, providing an estimate from below of the Bergman kernel 
on Kobayashi balls:

\begin{lemma}[\textbf{\cite[Lemma~3.2 and Corollary 3.3]{AS}}]  
\label{piu}
Let $D\subset\subset\C ^n$ be a bounded strongly
pseudoconvex domain. Then for every $r\in(0,1)$ there exist $c_r>0$ and 
$\delta_r>0$ such that 
if $z_0\in D$ satisfies $\delta(z_0)<\delta_r$ then
\[
\forall{z\in B_D(z_0,r)}\qquad \min\{|K(z,z_0)|, |k_{z_0}(z)|^2\}\ge\frac{c_r}{\delta(z_0)^{n+1}}\;.
\]
\end{lemma}

We end this section with an easy (but sometimes useful) lemma:

\begin{lemma}
\label{compwn}
Let $D\subset\subset\C ^n$ be a bounded domain, $p\in[1,+\infty]$ and $\beta_1\le\beta_2$. Then
$A^p(D,\beta_1)\hookrightarrow A^p(D,\beta_2)$ continuously, that is
\[
\|\cdot\|_{p,\beta_2}\preceq \|\cdot\|_{p,\beta_1}\;.
\]
\end{lemma}

\begin{proof}
Put $D_1=\{z\in D\mid \delta(z)\ge 1\}$ and $D_0=D\setminus D_1$. Assume $p<+\infty$, and take $f\in A^p(D,\beta_1)$. Then
\[
\begin{aligned}
\int_D |f|^p\delta^{\beta_2}\,d\nu&=\int_{D_0}|f|^p\delta^{\beta_2-\beta_1}\delta^{\beta_1}\,d\nu
+\int_{D_1}|f|^p\delta^{\beta_2-\beta_1}\delta^{\beta_1}\,d\nu\\
&\le \int_{D_0} |f|^p\delta^{\beta_1}\,d\nu+M^{\beta_2-\beta_1}\int_{D_1}|f|^p\delta^{\beta_1}\,d\nu\\
&\le\max(1,M^{\beta_2-\beta_1})\int_D |f|^p\delta^{\beta_1}\,d\nu
\end{aligned}
\]
where $M=\max_{z\in D}\delta(z)<+\infty$, and we are done in this case.

If $p=+\infty$ we instead have
\[
\|f\|_{\infty,\beta_2}=\|f\delta^{\beta_2}\|_\infty\le M^{\beta_2-\beta_1}\|f\delta^{\beta_1}\|_\infty=
M^{\beta_2-\beta_1}\|f\|_{\infty,\beta_1}\;,
\]
as claimed.
\end{proof}

\section{$\theta$-Carleson measures}

In this section we shall characterize Carleson measures for weighted Bergman spaces following ideas introduced in~\cite{AS} and~\cite{Ka}. Let us begin with:

\begin{definition}
\label{thetaCarl}
Let $D\subset\subset\C ^n$ be a bounded domain, $\beta$,~$\theta\in\R$ and $p\ge 1$. A (analytic)
\emph{Carleson measure} of~$A^p(D,\beta)$ is a finite positive Borel measure on~$D$ such that there is a continuous inclusion $A^p(D,\beta)\hookrightarrow L^p(\mu)$, that is there exists a constant $C>0$ such that
\[
\forall f\in A^p(D,\beta)\qquad\quad \int_D |f|^p\,d\mu\le C\|f\|^p_{p,\beta}\;.
\]
On the other hand, a (geometric) \emph{$\theta$-Carleson measure} is a finite positive Borel measure on~$D$ such that 
\[
\mu\bigl(B_D(\cdot,r)\bigr)\preceq \nu\bigl(B_D(\cdot,r)\bigr)^\theta
\]
for all $r\in(0,1)$, where the constant might depend on~$r$.
\end{definition}

\begin{remark}
A $1$-Carleson measure is just an usual (geometric) Carleson measure, and thus (by~\cite{AS})
a Carleson measure of all~$A^p(D)$ if $D$ is strongly pseudoconvex. 
Furthermore, every finite measure clearly is $\theta$-Carleson for any $\theta\le 0$. 
\end{remark}

At the end of this section we shall give examples of $\theta$-Carleson measures (see Example~\ref{es:deltaeta} and Theorem~\ref{dtre}); but first we
shall prove that $\theta$-Carleson measures and Carleson measures of weighted Bergman spaces are one and the same thing:

\begin{theorem}
\label{carthetaCarluno}
Let $D\subset\subset\C ^n$ be a bounded strongly
pseudoconvex domain, and choose $1-\frac{1}{n+1}<\theta<2$. Then the following assertions are
equivalent:
\begin{itemize}
\item[(i)] $\mu$ is a Carleson measure of all $A^p\bigl(D,(n+1)(\theta-1)\bigr)$, that is $A^p\bigl(D,(n+1)(\theta-1)\bigr)\hookrightarrow L^p(\mu)$ continuously for all
$p\in[1,+\infty)$;
\item[(ii)] there exists $p\in[1,+\infty)$ such that $\mu$ is a Carleson measure of $A^p\bigl(D,(n+1)(\theta-1)\bigr)$;
\item[(iii)] $\mu$ is $\theta$-Carleson;
\item[(iv)] there exists $r_0\in(0,1)$ such that $\mu\bigl(B_D(\cdot,r_0)\bigr)\preceq \nu\bigl(B_D(\cdot,r_0)\bigr)^\theta$;
\item[(v)] for every $r\in(0,1)$ and for every $r$-lattice $\{a_k\}$ in $D$ one has
\[
\mu\bigl(B_D(a_k,r)\bigr)\preceq \nu\bigl(B_D(a_k,r)\bigr)^\theta\;;
\]
\item[(vi)] there exists $r_0\in(0,1)$ and a $r_0$-lattice $\{a_k\}$ in $D$ such that
\[
\mu\bigl(B_D(a_k,r_0)\bigr)\preceq\nu\bigl(B_D(a_k,r_0)\bigr)^\theta\;.
\]
\end{itemize}
\end{theorem}

\begin{proof}
(i)$\Longrightarrow$(ii). Obvious.
\smallbreak

(ii)$\Longrightarrow$(iii). 
Fix $r\in(0,1)$, and let $\delta_r>0$ and $c_r>0$ be given by 
Lemma~\ref{piu}. We must prove that $\mu\bigl(B_D(z_0,r)\bigr)\le C\nu\bigl(B_D(z_0,r)\bigr)^\theta$ for all $z_0\in D$, where $C>0$ is a suitable constant independent of~$z_0$. Thanks to Lemma~\ref{sei}, it suffices to prove this statement when $\delta(z_0)<\delta_r$. 

Lemmas~\ref{piu} and~\ref{BKbasic} yield
\[
\begin{aligned}
\frac{c^p_r}{\delta(z_0)^{(n+1)p}}\mu\bigl(B_D(z_0,r)\bigr)&\le
\int_{B_D(z_0,r)}|k_{z_0}(\zeta)|^{2p}\,d\mu(\zeta) \le \int_D|k_{z_0}(\zeta)|^{2p}\,d\mu(\zeta)\\
&\preceq \int_D |k_{z_0}(\zeta)|^{2p} \delta(\zeta)^{(n+1)(\theta-1)}\,d\nu(\zeta)\\
&\preceq \delta(z_0)^{(n+1)p}\int_D|K(\zeta,z_0)|^{2p}\delta(\zeta)^{(n+1)(\theta-1)}\,d\nu(\zeta)\;.
\end{aligned}
\]
Since $-1< (n+1)(\theta-1)< n+1\le(n+1)(2p-1)$, we can apply
Theorem~\ref{BKp} obtaining
\[
\frac{c^p_r}{\delta(z_0)^{(n+1)p}}\mu\bigl(B_D(z_0,r)\bigr)\preceq \delta(z_0)^{(n+1)(\theta-p)}\;.
\]
Therefore
\[
\mu\bigl(B_D(z_0,r)\bigr)\preceq \delta(z_0)^{(n+1)\theta}\preceq \nu\bigl(B_D(z_0,r)\bigr)^\theta\;,
\]
where we used Lemma~\ref{sei}.
\smallbreak

(iii)$\Longrightarrow$(iv)$\Longrightarrow$(vi) and (iii)$\Longrightarrow$(v)$\Longrightarrow$(vi). Obvious.

\smallbreak
(vi)$\Longrightarrow$(i). 
Fix $p\in[1,+\infty)$, and take $f\in A^p\bigl(D,(n+1)(\theta-1)\bigr)$. Clearly we have
\[
\int_D |f(z)|^p\,d\mu(z)\le\sum_{k=0}^\infty \int_{B_D(a_k,r_0)}|f(z)|^p\,d\mu(z)\;.
\]
Now, Lemma~\ref{due} gives a $K>0$ depending only on~$r_0$ (and $D$) such that
\[
\begin{aligned}
\int_{B_D(a_k,r_0)}|f(z)|^p\,d\mu(z)&\le \frac{K}{\nu\bigl(B_D(a_k,r_0)\bigr)}\int_{B_D(a_k,r_0)}
\left[\int_{B_D(a_k,R_0)}|f(\zeta)|^p\,d\nu(\zeta)\right]d\mu(z)\\
&=K\,\frac{\mu\bigl(B_D(a_k,r_0)\bigr)}{\nu\bigl(B_D(a_k,r_0)\bigr)}\int_{B_D(a_k,R_0)}
|f(\zeta)|^p\,d\nu(\zeta)\\
&\preceq \nu\bigl(B_D(a_k,r_0)\bigr)^{\theta-1}\int_{B_D(a_k,R_0)}|f(\zeta)|^p\,d\nu(\zeta)\;,
\end{aligned}
\]
where $R_0=\frac{1}{2}(1+r_0)$. 
Now, Lemma~\ref{sei} yields $\nu\bigl(B_D(\cdot,r_0)\bigr)^{\theta-1}\preceq \delta^{(n+1)(\theta-1)}$
both when $\theta-1\ge 0$ and when $\theta-1\le 0$; therefore recalling Lemma~\ref{sette} we get
\[
\begin{aligned}
\int_{B_D(a_k,r_0)}|f(z)|^p\,d\mu(z)&\preceq \delta(a_k)^{(n+1)(\theta-1)}
\int_{B_D(a_k,R_0)}|f(\zeta)|^p\,d\nu(\zeta)\\
&\preceq \int_{B_D(a_k,R_0)}|f(\zeta)|^p\delta(\zeta)^{(n+1)(\theta-1)}\,d\nu(\zeta)\;.
\end{aligned}
\]
Summing on $k$ and recalling that, by definition of $r_0$-lattice, there is $m\in\N$ such that every point of~$D$ is contained in at most $m$ balls of the
form $B_D(a_k,R_0)$ we obtain
\[
\|f\|_{L^p(\mu)}^p=\int_D |f(z)|^p\,d\mu(z)\preceq \int_D |f(\zeta)|^p \delta(\zeta)^{(n+1)(\theta-1)}\,d\nu(\zeta)
=\|f\|^p_{p,(n+1)(\theta-1)}\;,
\]
as claimed. 
\end{proof}

%
%

\begin{remark}
\label{rem:uno}
The proof shows that the chains of implications
(iii)$\Longrightarrow$(iv)$\Longrightarrow$(vi)$\Longrightarrow$(i)$\Longrightarrow$(ii) 
and (iii)$\Longrightarrow$(v)$\Longrightarrow$(vi)$\Longrightarrow$(i)$\Longrightarrow$(ii) hold for all $\theta\in\R$,
and that the implication (ii)$\Longrightarrow$(iii) holds for $1-\frac{1}{n+1}<\theta<2p$. 
When $\theta>2p$ condition (ii) just implies that $\mu$ is $2p$-Carleson, and when $\theta=2p$ condition (ii) implies that $\mu$ is $(2p-\eps)$-Carleson for all $\eps>0$.
Furthermore, the proof shows that the norm of the 
inclusion in (i) is bounded by a constant independent of $p$, and also of $\theta$ if the latter 
is restricted to vary in a compact interval. Finally, the proof is somewhat new even
for $\theta=1$, because it does not depend on \cite{CM}.
\end{remark}

\begin{remark}
\label{rem:psh}
The proof of the implication (vi)$\Longrightarrow$(i), recalling Lemma~\ref{uno}, shows that \emph{if $\mu$ is $\theta$-Carleson then we have
\[
\int_D \chi(z)\,d\mu(z)\preceq \int_D \chi(\zeta)\delta(\zeta)^{(n+1)(\theta-1)}\,d\nu(\zeta)
\]
for all nonnegative plurisubharmonic functions $\chi\colon D\to\R^+$.}
\end{remark}

As anticipated in the introduction, another useful characterization of $\theta$-Carleson measures relies on the Berezin transform.

\begin{definition}
Let $\mu$ be a finite positive Borel measure on a bounded strongly pseudoconvex domain
$D\subset\subset\C^n$. The \emph{Berezin transform} of $\mu$ is the function $B\mu\colon D\to\R^+$ given by
\[
B\mu(z_0)=\int_D |k_{z_0}(z)|^2\,d\mu(z)\;.
\]
\end{definition}

Then:

\begin{theorem}
\label{carthetaCarldue}
Let $D\subset\subset\C ^n$ be a bounded strongly
pseudoconvex domain, and 
choose $\theta>0$. 
Then the following assertions are equivalent:
\begin{itemize}
\item[(i)] $\mu$ is $\theta$-Carleson;
\item[(ii)] $B\mu\preceq \delta^{(n+1)(\theta-1)}$.
\end{itemize}
\end{theorem}

\begin{proof}
(i)$\Longrightarrow$(ii). Using Theorems~\ref{carthetaCarluno} and~\ref{BKp} (and Remark~\ref{rem:uno}) we obtain
\[
B\mu(z_0)=\|k_{z_0}\|^2_{L^2(\mu)}\preceq \|k_{z_0}\|^2_{2,(n+1)(\theta-1)}\preceq
\delta(z_0)^{(n+1)(\theta-1)}\;,
\]
as claimed.
\smallskip

(ii)$\Longrightarrow$(i). Fix $r\in(0,1)$; we must show that $\mu\bigl(B_D(z_0,r)\bigr)\le C
\nu\bigl(B_D(z_0,r)\bigr)^\theta$ for all $z_0\in D$, where $C>0$ is a suitable constant
independent of~$z_0$.  Let $\delta_r>0$ and $c_r>0$ be given by 
Lemma~\ref{piu}; clealy it suffices to prove the claim for $\delta(z_0)<\delta_r$. We have
\[
B\mu(z_0)=\int_D |k_{z_0}(z)|^2\,d\mu(z)\ge \int_{B_D(z_0,r)}|k_{z_0}(z)|^2\,d\mu(z)
\ge \frac{c_r}{\delta(z_0)^{n+1}}\mu\bigl(B_D(z_0,r)\bigr)\;;
\]
therefore
\[
\mu\bigl(B_D(z_0,r)\bigr)\preceq \delta(z_0)^{n+1}B\mu(z_0)\preceq \delta(z_0)^{(n+1)\theta}
\]
by the hypothesis, and the assertion follows from Lemma~\ref{sei}.
\end{proof}


We end this section by giving examples of $\theta$-Carleson measures. The next lemma provides a way of shifting the value of~$\theta$:

\begin{lemma}
\label{changeCarl}
Let $D\subset\subset\C ^n$ be a bounded domain, and $\theta$,~$\eta\in\R$.
Then a finite positive Borel measure $\mu$ is $\theta$-Carleson if and only if
$\delta^\eta\mu$ is $(\theta+\frac{\eta}{n+1})$-Carleson.
\end{lemma}

\begin{proof}
Assume $\mu$ is $\theta$-Carleson, set $\mu_\eta=\delta^\eta\mu$, and choose $r\in(0,1)$.
Then Lemmas~\ref{sette} and~\ref{sei} yield
\[
\begin{aligned}
\mu_\alpha\bigl(B_D(z_0,r)\bigr)&=\int_{B_D(z_0,r)}\delta(\zeta)^\eta\,d\mu(\zeta)
\preceq \delta(z_0)^\eta \mu\bigl(B_D(z_0,r)\bigr)\\
&\preceq  \delta(z_0)^\eta \nu\bigl(B_D(z_0,r)\bigr)^\theta
\preceq \nu\bigl(B_D(z_0,r)\bigr)^{\theta+\frac{\eta}{n+1}}\;,
\end{aligned}
\]
and so $\mu_\eta$ is $\left(\theta+\frac{\eta}{n+1}\right)$-Carleson. Since $\mu=(\mu_\eta)_{-\eta}$, the converse follows too.
\end{proof}

\begin{example}
\label{es:deltaeta}
For instance, as anticipated in the introduction, $\delta^\eta\nu$ is $\left(1+\frac{\eta}{n+1}\right)$-Carleson.
\end{example}

To give another class of examples of $\theta$-Carleson measures, we recall the following definition:

\begin{definition}
Let $(X,d)$ be a metric space.  A sequence $\Gamma=\{x_j\}\subset
X$ is \emph{uniformly discrete} if there exists $\delta>0$
such that $d(x_j,x_k)\ge\delta$ for all $j\ne k$. In this case 
$\inf\limits_{j\ne k} d(x_j,x_k)$ is the {\sl separation constant} of~$\Gamma$. Furthermore,
given $x\in X$ and $r>0$ we shall denote by $N(x,r,\Gamma)$ the number of points $x_j\in\Gamma$ with $d(x_j,x)<r$.
\end{definition}

\begin{theorem} 
\label{dtre} 
Let $D\subset\subset\C^n$ be a bounded strongly pseudoconvex domain, considered as a metric space with the distance~$\rho_D=\tanh k_D$, and choose $1-\frac{1}{n+1}<\theta<2$. Let $\Gamma=\{z_j\}_{j\in\N}$ be a sequence
in~$D$. Then $\Gamma$ is a finite union of uniformly discrete sequences if and only if $\sum_j \delta(z_j)^{(n+1)\theta}\delta_{z_j}$ is a $\theta$-Carleson measure, where $\delta_{z_j}$ is the Dirac measure in~$z_j$.
\end{theorem}

\begin{proof}
Clearly it suffices to prove the only if part when $\Gamma$ is a single
uniformly discrete sequence. Choose $p>\max\{1,\theta/2\}$, and
let $2r>0$ be the separation constant
of~$\Gamma$.
By the triangle inequality, the Kobayashi balls $B_D(z_j,r)$ are
pairwise disjoint. Hence for any $f\in A^p\bigl(D,(n+1)(\theta-1)\bigr)$ Lemma~\ref{sette} yields
\[
\begin{aligned}
\int_D|f(z)|^p\delta(z)^{(n+1)(\theta-1)}\,d\nu(z)
&\ge\sum_{z_j\in\Gamma}\int_{B_D(z_j,r)}|f(z)|^p \delta(z)^{(n+1)(\theta-1)}\,d\nu(z)\\
&\succeq \sum_{z_j\in\Gamma}\delta(z_j)^{(n+1)(\theta-1)}\int_{B_D(z_j,r)}|f(z)|^p \,d\nu(z)
\;.
\end{aligned}
\]
Now, $|f|^p$ is plurisubharmonic and nonnegative; hence Lemma~\ref{due} 
and Lemma~\ref{sei} yield
\[
\int_D|f(z)|^p\delta(z)^{(n+1)(\theta-1)}\,d\nu(z)
\succeq \sum_{z_j\in\Gamma}\delta(z_j)^{(n+1)(\theta-1)}\delta(z_j)^{n+1}|f(z_j)|^p
=\sum_{z_j\in\Gamma}\delta(z_j)^{(n+1)\theta}|f(z_j)|^p
\]
and the assertion follows from Theorem~\ref{carthetaCarluno} and Remark~\ref{rem:uno}.

Assume conversely that $\mu=\sum_j \delta(z_j)^{(n+1)\theta}\delta_{z_j}$ is a $\theta$-Carleson measure. \cite[Lemma~4.1]{AS} shows that it suffices to prove that $\sup_{z_0\in D} N(z_0,r, \Gamma)<+\infty$, for any $r\in(0,1)$. Fix $r\in(0,1)$, and let $\delta_r>0$ be given
by Lemma~\ref{piu}. By Lemma~\ref{sette}, if $\delta(z_0)\ge\delta_r$ then $w\in B_D(z_0,r)$ implies
$\delta(w)\succeq \delta_r$. It is easy to see that, since $\mu$ should be a finite measure, only a finite number of~$z_j\in\Gamma$ can have $\delta(z_j)\succeq\delta_r$; therefore to get the assertion it suffice to prove
that the supremum is finite when $\delta(z_0)<\delta_r$.

Given $z_0\in D$ with $\delta(z_0)<\delta_r$, Lemma~\ref{piu} yields
\[
\forall{z\in B_D(z_0,r)}\ \ \ \ 
\delta(z_0)^{n+1}|k_{z_0}(z)|^2\ge  c_r\;.
\]
Hence using again Lemma~\ref{sette} we obtain
\begin{eqnarray*}
N(z_0,r,\Gamma) &\le&\frac{1}{c_r}
\sum_{z\in B_D(z_0,r)\cap\Gamma}\delta(z_0)^{n+1}|k_{z_0}(z)|^2\\
&\preceq&\delta(z_0)^{(n+1)(1-\theta)}\sum_{z\in B_D(z_0,r)\cap\Gamma}\delta(z)^{(n+1)\theta}|k_{z_0}(z)|^2\le\delta(z_0)^{(n+1)(1-\theta)}\|k_{z_0}\|^2_{L^2(\mu)}\\
&\preceq&\delta(z_0)^{(n+1)(1-\theta)} \|k_{z_0}\|_{2,(n+1)(\theta-1)}^2\\
&\preceq& 1
\end{eqnarray*}
by Theorems~\ref{BKp} and~\ref{carthetaCarluno} (and Remark~\ref{rem:uno}), as desired.
\end{proof}

\begin{remark}
\label{rem:impl}
Notice that the proof that if $\Gamma$ is a finite union of uniformly discrete sequences then $\sum_j \delta(z_j)^{(n+1)\theta}\delta_{z_j}$ is a $\theta$-Carleson measure works for
any $\theta>0$. 
\end{remark}

\section{Vanishing $\theta$-Carleson measures}

In this section we shall characterize vanishing Carleson measures for weighted Bergman spaces; along the way we shall prove a few results on the functional analysis of weighted Bergman spaces that shall be useful in the next section too.

\begin{definition}
\label{vanthetaCarl}
Let $D\subset\subset\C ^n$ be a bounded domain, $\beta$,~$\theta\in\R$ and $p\ge 1$. A (analytic)
\emph{vanishing Carleson measure} of~$A^p(D,\beta)$ is a finite positive Borel measure on~$D$ such that there is a compact inclusion $A^p(D,\beta)\hookrightarrow L^p(\mu)$.

On the other hand, a (geometric) \emph{vanishing $\theta$-Carleson measure} is a finite positive Borel measure on~$D$ such that 
\[
\lim_{z_0\to\de D} \frac{\mu\bigl(B_D(z_0,r)\bigr)}{\nu\bigl(B_D(z_0,r)\bigr)^\theta}=0
\]
for all $r\in(0,1)$.
\end{definition}

\begin{remark}
\label{rem:carvancar}
In particular, every $\theta$-Carleson measure is a vanishing $\theta'$-Carleson measure
for all $\theta'<\theta$. For instance, a Carleson measure is a vanishing $\theta$-Carleson measure for all $\theta<1$. 
\end{remark}

%

We start with an easy generalization of a standard lemma (see, e.g., \cite[Lemma~1.4.1]{Kr}):

\begin{lemma}
\label{th:basiclemma}
Let $D\subset\subset\C^n$ be a bounded domain, $p\in[1,+\infty]$ and $\beta\in\R$. Then
for every relatively compact subdomain $D_0\subset\subset D$ we can find a constant
$C=C(D_0,p,\beta)>0$ such that
\[
\sup_{z\in D_0}|f(z)|\le C\|f\|_{p,\beta}
\] 
for all $f\in A^p(D,\beta)$.
\end{lemma}

\begin{proof}
Given $r>0$ and $z\in\C^n$, we shall denote by $B_r(z)$ the Euclidean ball of radius~$r$ and
center~$z$. If $r_0=\inf_{z\in D_0}\delta(z)>0$ and $M=\sup_{z\in D}\delta(z)<+\infty$, we have
\[
\frac{r_0}{2}\le\delta(\zeta)\le M
\]
for all $\zeta\in B_{r_0/2}(z)$ and $z\in D_0$. 

Assume that $p\in[1,+\infty)$. Using the usual submean property for nonnegative plurisubharmonic functions, for all $z\in D_0$ by H\"older's inequality we then have  
\[
\begin{aligned}
|f(z)|&\le\frac{1}{\nu\bigl(B_{r_0/2}(z)\bigr)}\int_{B_{r_0/2}(z)}|f(\zeta)|\,d\nu(\zeta)\cr
 &\le \nu\bigl(B_{r_0/2}(z)\bigr)^{\frac{1}{q}-1}\left[
\int_{B_{r_0/2}(z)}|f(\zeta)|^p\,d\nu(\zeta)\right]^{1/p}\cr
&\le \nu\bigl(B_{r_0/2}(z)\bigr)^{\frac{1}{q}-1}\max\left\{M,\frac{2}{r_0}\right\}^{|\beta|/p}
\left[\int_{B_{r_0/2}(z)}|f(\zeta)|^p\delta(\zeta)^\beta\,d\nu(\zeta)\right]^{1/p}\cr
&\le C(D_0,p,\beta) \|f\|_{p,\beta}\;,
\end{aligned}
\]
where $q$ is the conjugate exponent of~$p$, and we are done. 

Finally, if $p=+\infty$ we have
\[
\sup_{z\in D_0}|f(z)|\le \max\left\{M,\frac{1}{r_0}\right\}^{|\beta|}\|f\|_{\infty,\beta}\;,
\]
and we are done in this case too.
\end{proof}

Using this we obtain a basic compactness property for weighted Bergman spaces on bounded domains:

\begin{lemma}
\label{relcomp}
Let $D\subset\subset\C^n$ be a bounded domain, $1\le p\le\infty$ and $\beta\in\R$. Then:
\begin{itemize}
\item[(i)] if $\{f_k\}\subset A^p(D,\beta)$ is a norm-bounded sequence converging uniformly
on compact subsets to $h\in\mathcal{O}(D)$, then $h\in A^p(D,\beta)$;
\item[(ii)] the inclusion $A^p(D,\beta)\hookrightarrow
\mathcal{O}(D)$ is compact, that is, any norm-bounded subset of $A^p(D,\beta)$ is relatively compact
in~$\mathcal{O}(D)$.
\end{itemize}
\end{lemma}

\begin{proof}
(i) If $p=\infty$ the assertion is trivial; let then $1\le p<\infty$ and assume that $\{f_k\}\subset A^p(D,\beta)$ is a norm-bounded sequence converging
uniformly on compact subsets to $h\in\mathcal{O}(D)$. Then
\[
\int_D |h|^p\delta^\beta\,d\nu=\int_D\lim_{k\to\infty}|f_k|^p\delta^\beta\,d\nu\le
\liminf_{k\to\infty}\int_D |f_k|^p\delta^\beta\,d\nu\le\sup_{k}\|f_k\|^p_{p,\beta}\;,
\]
by Fatou's lemma, and thus $h\in A^p(D,\beta)$ as claimed.

(ii) We have to prove that any norm-bounded sequence in $A^p(D,\beta)$ admits a subsequence
converging uniformly on compact subsets. But indeed, Lemma~\ref{th:basiclemma}
says that the sup-norm on a relatively compact subset $D_0\subset\subset D$ of any
$f\in A^p(D,\beta)$ is bounded by a constant times its $A^p(D,\beta)$-norm. 
So if $\{f_k\}\subset A^p(D,\beta)$ is norm-bounded, by taking a countable increasing
exhaustion of $D$ by relatively compact subdomains and applying Montel's theorem to each subdomain, we obtain a subsequence $\{f_{k_j}\}$ converging uniformly on compact subsets 
to a holomorphic function~$h\in\mathcal{O}(D)$ --- and actually $h\in A^p(D,\beta)$, by (i). 
\end{proof}

As a consequence we obtain the following characterization of vanishing Carleson measures
of~$A^p(D,\beta)$:

\begin{lemma}
\label{weakly}
Let $D\subset\subset\C^n$ be a bounded domain, and $\mu$ a finite positive Borel measure on~$D$. Take $1\le p\le\infty$ and $\beta\in\R$. Then $\mu$ is a vanishing Carleon measure of~$A^p(D,\beta)$ if and only if $\|f_k\|_{L^p(\mu)}\to 0$ for all norm-bounded sequences $\{f_k\}\subset A^p(D,\beta)$ converging to~$0$ uniformly on compact subsets. 
\end{lemma}

\begin{proof}
Assume that the inclusion $A^p(D,\beta)\hookrightarrow L^p(\mu)$ is compact, and take $\{f_k\}\subset A^p(D,\beta)$ norm-bounded and converging to~$0$ uniformly on compact subsets. In particular, $\{f_k\}$ is relatively compact in~$L^p(\mu)$;
we must prove that $f_k\to 0$ in~$L^p(\mu)$. To do so, by compactness, it suffices to show that
$0$ is the unique limit point of $\{f_k\}$ in~$L^p(\mu)$. Let $\{f_{k_j}\}$ be a subsequence
converging to $h\in L^p(\mu)$. Passing if necessary to
a subsequence 
we can assume that 
$f_{k_j}(z)\to h(z)$ $\mu$-almost everywhere. But $f_k\to 0$ uniformly on compact
subsets; therefore $h\equiv 0$ and we are done.

Conversely, assume that all norm-bounded sequences in $A^p(D,\beta)$ converging to~$0$
uniformly on compact subsets
converge to~0 in~$L^p(\mu)$. To prove that the inclusion $A^p(D,\beta)\hookrightarrow A^p(\mu)$ is compact it
suffices to show that if $\{f_k\}$ is norm-bounded in~$A^p(D,\beta)$ then it admits a subsequence
converging in~$L^p(\mu)$. Lemma~\ref{relcomp} yields a subsequence $\{f_{k_j}\}$
converging uniformly on compact subsets to $h\in A^p(D,\beta)$. Then $\{f_{k_j}-h\}$
converges to~$0$ uniformly on compact subsets; by assumption, this yields $\|f_{k_j}-h\|_{L^p(\mu)}\to 0$, and thus $f_{k_j}\to h$ in~$L^p(\mu)$, as desired.
\end{proof}

We shall also need the following characterization of weakly convergent sequences in $A^p(D,\beta)$ for $1<p<\infty$:

\begin{lemma}
\label{ucsweak}
Let $D\subset\subset\C^n$ be a bounded domain, $1< p<\infty$ and $\beta\in\R$. Then:
\begin{itemize}
\item[(i)] a sequence
$\{f_k\}\subset A^p(D,\beta)$ is norm-bounded and converges uniformly on compact subsets to $h\in A^p(D,\beta)$ if and only if
it converges weakly to~$h$;
\item[(ii)] the unit ball of $A^p(D,\beta)$ is weakly compact, and thus $A^p(D,\beta)$ is reflexive.
\end{itemize}
\end{lemma}

\begin{proof}
(i) Without loss of generality we can assume that $h\equiv 0$. Assume that $\{f_k\}$ is norm-bounded
and converges  
uniformly on compact subsets to~$0$;
we have to prove that
$\Phi(f_k)\to 0$ for all $\Phi\in A^p(D,\beta)^*$. Take $\Phi\in A^p(D,\beta)^*$; by the Hahn-Banach theorem we can find
$\hat\Phi\in L^p(\delta^\beta\nu)^*$ such that $\hat\Phi|_{A^p(D,\beta)}=\Phi$. By the Riesz representation theorem
we then get $g\in L^q(\delta^\beta\nu)$ such that 
\[
\Phi(f)=\int_D fg\delta^\beta\,d\nu
\]
for all $f\in A^p(D,\beta)$, where $q$ is the conjugate exponent of~$p$. So it suffices to prove that
$\int_D f_k g\delta^\beta\,d\nu\to 0$ for all $g\in L^q(\delta^\beta\nu)$; since functions with compact support are dense 
in~$L^q(\delta^\beta\nu)$ it suffices to prove this when $g$ has compact support. But in that case,
denoting by $\textrm{Vol}_\beta\bigl(\textrm{supp}(g)\bigr)$ the volume of $\textrm{supp}(g)$ 
with respect to the measure~$\delta^\beta\nu$, we have
\begin{eqnarray*}
\left|\int_D f_kg\delta^\beta\,d\nu\right|&\le&\int_{\mathrm{supp}(g)}|f_kg|\delta^\beta\,d\nu
\le \left[\int_{\mathrm{supp}(g)}|f_k|^p\delta^\beta\,d\nu\right]^{1/p}\|g\|_{L^q(\delta^\beta\nu)}\\
&\le&\textrm{Vol}_\beta\bigl(\textrm{supp}(g)\bigr)\sup_{z\in\mathrm{supp}(g)}|f_k(z)|\,\|g\|_{L^q(D,\beta)}
\to 0
\end{eqnarray*}
because $f_k\to 0$ uniformly on compact subsets, and we are done.

Conversely, assume that $f_k\to 0$ weakly in $A^p(D,\beta)$; in particular, 
is norm-bounded in~$A^p(D,\beta)$. Therefore,
thanks to Lemma~\ref{relcomp}.(ii), to prove that $f_k\to 0$ uniformly on compact subsets
it suffices to show that any converging (uniformly on compact subsets) subsequence
must converge to~0. But if $f_{k_j}\to h\in A^p(D,\beta)$ uniformly on compact subsets the previous
argument shows that $f_{k_j}$ converges weakly to~$h$; the uniqueness of the weak limit
then yields $h\equiv 0$, and we are again done.

(ii) Let $\{g_k\}$ be a sequence in the unit ball of $A^p(D,\beta)$. By Lemma~\ref{relcomp}.(ii),
there is a subsequence $\{g_{k_j}\}$ converging uniformly on compact subsets to~$g\in\mathcal{O}(D)$; furthermore, Lemma~\ref{relcomp}.(i) yields $g\in A^p(D,\beta)$. But then part~(i) 
implies that $g_{k_j}\to g$ weakly in $A^p(D,\beta)$, and we are done. 
\end{proof}

Thus, for $1<p<\infty$, Lemma \ref{weakly} is a particular case of the following (well-known) proposition:

\begin{proposition}
\label{th:gencompact}
Let $T\colon X\to Y$ be a linear operator between Banach spaces. Then:
\begin{itemize}
\item[(i)]
if $T$ is compact, then for any 
sequence $\{x_k\}\subset  X$ weakly converging 
to $0$ the sequence $\{Tx_k\}$ strongly converges to $0$ in $Y$;
\item[(ii)] assume that the unit ball of $X$ is weakly compact; then if for any  
sequence $\{x_k\}\subset  X$ weakly converging  
to $0$ the sequence $\{Tx_k\}$ strongly converges to $0$ in $Y$ it follows that $T$ is compact.
\end{itemize}
\end{proposition}

\begin{proof}
(i)
Suppose $T$ is compact and let $\{x_k\}\subset X$ be 
weakly converging to~$0$. If, by contradiction, $\|Tx_k\|_Y$ does not converge to $0$ then, up to passing to a subsequence, we may assume that there is $\delta>0$ such that $\|Tx_k\|_Y\ge \delta$ for all~$k$. Since $T$ is compact, there are $y\in Y$ with $y\ne 0$ and a subsequence $\{x_{k_j}\}$ such that $\|Tx_{k_j} -y\|_Y\to 0$. In particular, $Tx_{k_j} \to y$ weakly in $Y$. Since for any $\psi\in Y^*$ we have $\psi \circ T \in X^*$, we obtain
\[
\psi (Tx_{k_j}) = (\psi\circ T)(x_{k_j})\to 0\;,
\]   
and thus $Tx_{k_j} \to 0$ weakly. It follows that $y=0$, contradicting the assumption.
\smallskip

(ii)
Suppose, by contradiction, that $T$ is not compact; then there exists a sequence $\{x_k\}\subset X$ in the unit ball such that $\{Tx_k\}$ has no strongly convergent subsequence. Now, by assumption the unit ball of $X$ is weakly compact;
therefore we can find a subsequence $\{x_{k_j}\}$ weakly convergent to~$x\in X$. 
%
%
Therefore the sequence $\{x_{k_j} - x\}$ converges to $0$ weakly, and thus, again by assumption,
the sequence $\{Tx_{k_j}-Tx\}$ converges to~$0$ strongly in~$Y$, that is $Tx_{k_j}\to Tx$ strongly, contradicting the choice of $\{x_k\}$. 
\end{proof}

\begin{corollary}
\label{th:4.6}
Let $D\subset\subset\C^n$ be a bounded domain, and take $p\in(1,+\infty)$ and $\beta\in\R$. Then a linear operator $T\colon A^p(D,\beta)\to X$ taking values in a Banach space $X$ is compact if and only if for any norm-bounded 
sequence $\{f_k\}\subset  A^p(D,\beta)$ 
converging uniformly on compact subsets 
to $0$ the sequence $\{Tf_k\}$ converges to $0$ in $X$.
\end{corollary}

\begin{proof}
It follows immediately from Proposition~\ref{th:gencompact} and Lemma~\ref{ucsweak}.
\end{proof}

For $p=1$ or $p=\infty$ we do not have such a general statement. However, for our needs
the following particular case will be enough:

\begin{lemma}
\label{th:cpt1infty}
Let $\mu$ be a finite positive Borel measure on a topological space $X$, and $1\le r\le\infty$. 
Assume that
$R\colon E\to L^r(\mu)$ is a compact operator, where $E$ is a Banach space. Then for every norm-bounded sequence $\{f_k\}\subset E$ such that $Rf_k(x)\to 0$ for $\mu$-almost every $x\in X$ we have $\|Rf_k\|_{L^r(\mu)}\to 0$.
\end{lemma}

\begin{proof}
Since $R$ is compact and $\{f_k\}$ is norm-bounded, the sequence $\{Rf_k\}$ is relatively
compact in~$L^r(\mu)$. If, by contradiction, $\|Rf_k\|_{L^r(\mu)}$ does not converge to~0,
up to a subsequence we can assume there is $\eps>0$ such that $\|Rf_k\|_{L^r(\mu)}\ge\eps$
for all~$k$. By compactness, there is a subsequence $\{Rf_{k_j}\}$ such that
$Rf_{k_j}\to h\in L^r(\mu)$ strongly. Passing, if necessary, to a subsubsequence we have
$Rf_{k_j}(x)\to h(x)$ for $\mu$-almost every $x\in X$; but then
the assumption forces $h\equiv 0$ and thus $\|Rf_{k_j}\|_{L^r(\mu)}\to 0$, contradiction.
\end{proof}

%

We can now prove a geometrical characterization of vanishing Carleson measures
of weighted Bergman spaces
in bounded strongly pseudoconvex domains, which is new even for~$A^p(D)$:

\begin{theorem} 
\label{Carleson} 
Let $\mu$ be a finite positive Borel measure on a bounded strongly pseudoconvex domain $D\subset\subset\C ^n$, and choose $1-\frac{1}{n+1}<\theta<2$. Then the following statements are equivalent:
\begin{itemize}
\item[(i)] $\mu$ is a vanishing Carleson measure of $A^p\bigl(D,(n+1)(\theta-1)\bigr)$ for all $p\in[1,+\infty)$;
\item[(ii)] $\mu$ is a vanishing Carleson measure of $A^p\bigl(D,(n+1)(\theta-1)\bigr)$ for some $p\in[1,+\infty)$;
\item[(iii)] $\mu$ is a vanishing $\theta$-Carleson measure;
\item[(iv)] there exists $r_0\in(0,1)$ such that 
\[
\lim_{z_0\to\de D}\frac{\mu\bigl(B_D(z_0,r_0)\bigr)}{\nu\bigl(B_D(z_0,r_0)\bigr)^\theta}=0\;;
\]
\item[(v)] for every $r\in(0,1)$ and for every $r$-lattice $\{a_k\}$ in $D$ one has
\[
\lim_{k\to+\infty}\frac{\mu\bigl(B_D(a_k,r)\bigr)}{\nu\bigl(B_D(a_k,r)\bigr)^\theta}=0\;;
\]
\item[(vi)] there exists $r_0\in(0,1)$ and a $r_0$-lattice $\{a_k\}$ in $D$ such that
\[
\lim_{k\to+\infty}\frac{\mu\bigl(B_D(a_k,r_0)\bigr)}{\nu\bigl(B_D(a_k,r_0)\bigr)^\theta}=0\;.
\]
\end{itemize}
\end{theorem}

\begin{proof}
(i)$\Longrightarrow$(ii), 
(iii)$\Longrightarrow$(iv), (iii)$\Longrightarrow$(v), (iv)$\Longrightarrow$(vi) and (v)$\Longrightarrow$(vi) are obvious.\smallbreak

(vi)$\Longrightarrow$(i) Fix $p\in[1,+\infty)$, and 
assume that $\{f_l\}\subset A^p\bigl(D,(n+1)(\theta-1)\bigr)$ is a norm-bounded sequence converging to $0$ uniformly on compact subsets; by Lemma~\ref{weakly} we must prove that $\|f_l\|_{L^p(\mu)}\to 0$.

Let $M>0$ be such that 
\begin{equation}
\|f_l\|_{p,(n+1)(\theta-1)} \le M
\label{unif}
\end{equation}
for all $l\in\N$.
By assumption, 
for any given $\eps>0$ there is $N_\eps\in\mathbb{N}$ such that
\begin{equation}
\forall k\geq N_\eps\qquad \frac{\mu\bigl(B_D(a_k,r_0)\bigr)}{\nu\bigl(B_D(a_k,r_0)\bigr)^\theta}< \eps\;.
\label{<e}
\end{equation}
Since the balls $B_D(a_k,r_0)$ cover $D$, it holds
\begin{equation}
\int_D |f_l|^p\,d\mu\ \leq\ \sum_{k=0}^\infty \int_{B_D(a_k,r_0)}|f_l(z)|^p\,d\mu(z)\;.
\label{stima1}
\end{equation}
Since $|f_l|^p$ is plurisubharmonic and nonnegative, Lemmas~\ref{due}, \ref{sei} and~\ref{sette} yield
\begin{eqnarray}
\nonumber \int_{B_D(a_k,r_0)}|f_l(z)|^p\,d\mu(z)&\le& {\frac{K_{r_0}}{\nu\bigl(B_D(a_k,r_0)\bigr)}}\int_{B_D(a_k,r_0)}d\mu(z)
\int_{B_D(a_k,R_0)}|f_l(\zeta)|^p\,d\nu(\zeta)\\
&=&K_{r_0}\frac{\mu\bigl(B_D(a_k,r_0)\bigr)}{\nu\bigl(B_D(a_k,r_0)\bigr)^\theta}
\nu\bigl(B_D(a_k,r_0)\bigr)^{\theta-1}
\int_{B_D(a_k,R_0)}|f_l(\zeta)|^p\,d\nu(\zeta)\nonumber\\
&\le&C_{r_0}\frac{\mu\bigl(B_D(a_k,r_0)\bigr)}{\nu\bigl(B_D(a_k,r_0)\bigr)^\theta}
\delta(a_k)^{(n+1)(\theta-1)}\int_{B_D(a_k,R_0)}|f_l(\zeta)|^p\,d\nu(\zeta)
\nonumber\\
&\le&\tilde C_{r_0}\frac{\mu\bigl(B_D(a_k,r_0)\bigr)}{\nu\bigl(B_D(a_k,r_0)\bigr)^\theta}
\int_{B_D(a_k,R_0)}|f_l(\zeta)|^p\delta(\zeta)^{(n+1)(\theta-1)}\,d\nu(\zeta)\;,
\label{stima2}
\end{eqnarray}
where $R_0={\frac12}(1+r_0)$ and $C_{r_0}$, $\tilde C_{r_0}>0$ are constants depending
only on~$r_0$.

Uniform convergence of $f_l$ to~$0$ on compact subsets implies that there exists $L_\varepsilon\in\N$ such that
\begin{equation}
\forall l\geq L_\eps \qquad \sum_{k=0}^{N_\eps-1} \int_{B_D(a_k,r_0)}|f_l(z)|^p\,d\mu(z)<\eps\,.
\label{stima3}
\end{equation}
Hence (\ref{stima1}), (\ref{stima2}), (\ref{stima3}), (\ref{<e}) and (\ref{unif}) yield 
\begin{eqnarray*}
\int_D |f_l|^p\,d\mu &\leq& \sum_{k=0}^{N_\eps-1} \int_{B_D(a_k,r_0)}|f_l(z)|^p\,d\mu(z) + \sum_{k=N_\eps}^{\infty} \int_{B_D(a_k,r_0)}|f_l(z)|^p\,d\mu(z)\\
&\le& \eps + \sum_{k=N_\eps}^\infty \tilde C_{r_0}\frac{\mu\bigl(B_D(a_k,r_0)\bigr)}{\nu\bigl(B_D(a_k,r_0)\bigr)^\theta}
\int_{B_D(a_k,R_0)}|f_l(\zeta)|^p\delta(\zeta)^{(n+1)(\theta-1)}\,d\nu(\zeta)\\
&\le& \eps + \tilde C_{r_0}\eps m\int_D|f_l(\zeta)|^p\delta(\zeta)^{(n+1)(\theta-1)}\,d\nu(\zeta) \le \eps(1+\tilde C_{r_0} m M^p)
\end{eqnarray*}
for a suitable $m\in\N$ as soon as $l\ge L_\eps$. 
Thus
\[
\lim_{l\to\infty}\int_D |f_l(z)|^p\,d\mu(z)\ =\ 0\,,
\]
and $\mu$ is a vanishing Carleson measure of $A^p\bigl(D,(n+1)(\theta-1)\bigr)$ by Lemma~\ref{weakly}.
\smallbreak

(ii)$\Longrightarrow$(iii) First of all notice that $\|k^2_{z_0}\|_{p,(n+1)(\theta-1)}=\|k_{z_0}\|^2_{2p,(n+1)(\theta-1)}$. Therefore the assumption on~$\theta$ and Theorem~\ref{BKp} imply
\[
\delta(z_0)^{(n+1)\left(1-\frac{\theta}{p}\right)}\|k^2_{z_0}\|_{p,(n+1)(\theta-1)}
\preceq \delta(z_0)^{(n+1)\left(1-\frac{\theta}{p}\right)}\delta(z_0)^{n+1-2\frac{n+1}{(2p)'}+2\frac{(n+1)(\theta-1)}{2p}}= 1\;,
\]
where $(2p)'$ is the conjugate exponent of~$2p>1$. Thus the family
$\{\delta(z_0)^{(n+1)\left(1-\frac{\theta}{p}\right)}k^2_{z_0}\}_{z_0\in D}$ is norm-bounded
in $A^p\bigl(D,(n+1)(\theta-1)\bigr)$, and then
Lemmas~\ref{BK} and \ref{weakly} imply that
\[
\lim_{z_0\to\de D}\|\delta(z_0)^{(n+1)\left(1-\frac{\theta}{p}\right)}k_{z_0}^2\|_{L^p(\mu)}=0\;.
\] 
Now choose $r\in(0,1)$, and let
$\delta_r>0$ be given by Lemma~\ref{piu}. Then if $\delta(z_0)<\delta_r$ we have
\begin{eqnarray*}
\|\delta(z_0)^{(n+1)\left(1-\frac{\theta}{p}\right)}k_{z_0}^2\|_{L^p(\mu)}^p&=&\delta(z_0)^{(n+1)(p-\theta)}\int_D |k_{z_0}(z)|^{2p}\, d\mu(z)\\
&\ge&\delta(z_0)^{(n+1)(p-\theta)}\int_{B_D(z_0,r)}
|k_{z_0}(z)|^{2p}\, d\mu(z)\\
&\succeq&\delta(z_0)^{-(n+1)\theta}\,\mu\bigl(B_D(z_0,r)\bigr)\\
&\succeq&\frac{\mu\bigl(B_D(z_0,r)\bigr)}{\nu\bigl(B_D(z_0,r)\bigr)^\theta}
\end{eqnarray*}
by Lemma~\ref{sei}, and we are done. 
\end{proof}

\begin{remark}
\label{rem:C}
The implications (iii)$\Longrightarrow$(iv)$\Longrightarrow$(vi)$\Longrightarrow$(i)$\Longrightarrow$(ii), as well as (iii)$\Longrightarrow$(v)$\Longrightarrow$(vi), hold for all
$\theta>0$. The implication (ii)$\Longrightarrow$(iii) works for $1-\frac{1}{n+1}<\theta<2p$;
when $\theta\ge 2p$ condition (ii) implies that $\mu$ is vanishing $(2p-\eps)$-Carleson 
for all~$\eps>0$. 
\end{remark}

\begin{corollary}
\label{th:CvC}
Let $\mu$ be a finite positive Borel measure on a bounded strongly pseudoconvex domain $D\subset\subset\C ^n$. If $1-\frac{1}{n+1}<\theta<2$ is such that there is a continuous
inclusion $A^p\bigl(D,(n+1)(\theta-1)\bigr)\hookrightarrow L^p(\mu)$ for some $p\in[1,+\infty)$
then the inclusion  $A^p\bigl(D,(n+1)(\theta'-1)\bigr)\hookrightarrow L^p(\mu)$ is compact
for all $p\in[1,+\infty)$ and $\theta'<\theta$.
\end{corollary}

\begin{proof}
It follows immediately from Theorem~\ref{Carleson} and Remark~\ref{rem:carvancar}.
\end{proof}

We can also use the Berezin transform to characterize vanishing $\theta$-Carleson measures:

\begin{theorem}
\label{carvanthetaCarldue}
Let $D\subset\subset\C ^n$ be a bounded strongly
pseudoconvex domain, and 
choose $1-\frac{1}{n+1}<\theta<2$. 
Then the following assertions are equivalent:
\begin{itemize}
\item[(i)] $\mu$ is vanishing $\theta$-Carleson;
\item[(ii)]  the Berezin transform $B\mu$ of $\mu$ satisfies $\delta(z_0)^{(n+1)(1-\theta)}B\mu(z_0)\to0$ as $z_0\to\de D$.
\end{itemize}
\end{theorem}

\begin{proof}
(i)$\Longrightarrow$(ii) 
Theorem~\ref{BKp} and  Lemma~\ref{BK} imply that
$\delta(z_0)^{(n+1)(1-\theta)/2}k_{z_0}$ is norm-bounded in $A^2\bigl(D,(n+1)(\theta-1)\bigr)$ and converges to $0$ uniformly on compact subsets as $z_0\to\partial D$.  So Theorem~\ref{Carleson} and Lemma~\ref{weakly} imply 
\[
\delta(z_0)^{(n+1)(1-\theta)}B\mu(z_0)=\delta(z_0)^{(n+1)(1-\theta)}\int_D|k_{z_0}(z)|^2\,d\mu=
 \|\delta(z_0)^{(n+1)(1-\theta)/2}k_{z_0}\|^2_{L^2(\mu)}\to 0
\]
as $z_0\to\de D$, as desired.
\smallskip

(ii)$\Longrightarrow$(i). Fix $r\in(0,1)$, and let $\delta_r>0$ be given by Lemma~\ref{piu}. Since we are only interested in the limit as $z_0$ goes to the boundary of $D$, we may assume $\delta(z_0)<\delta_r$. 
Then Lemma~\ref{piu} yields
\[
\frac{c_r}{\delta(z_0)^{n+1}}\mu\bigl(
B_D(z_0,r)\bigr)\le \int_{B_D(z_0,r)}|k_{z_0}(z)|^2\,d\mu(z)\le\int_{D}|k_{z_0}(z)|^2\,d\mu(z)= B\mu(z_0)\;.
\]
Recalling Lemma~\ref{sei} we get
\[
\mu\bigl(B_D(z_0,r)\bigr)\preceq \delta(z_0)^{n+1}B\mu(z_0)\preceq 
\delta(z_0)^{(n+1)(1-\theta)}B\mu(z_0)\nu\bigl(B_D(z_0,r)\bigr)^\theta\;;
\]
hence
\[
\frac{\mu\bigl(B_D(z_0,r)\bigr)}{\nu\bigl(B_D(z_0,r)\bigr)^\theta}\preceq \delta(z_0)^{(n+1)(1-\theta)}B\mu(z_0) \;,
\]
and (i) follows.
\end{proof}

\begin{remark}
\label{rem:Cb}
The implication (ii)$\Longrightarrow$(i) holds for all~$\theta>0$. On the other hand, condition (i) when $\theta\ge 2$ implies that $\delta(z_0)^{\eps-(n+1)}B\mu(z_0)\to 0$ as $z_0\to\de D$ for all $\eps>0$.
\end{remark}

The analogue of Lemma~\ref{changeCarl} is:

\begin{lemma}
\label{vCshift}
Let $D\subset\C ^n$ be a bounded domain, and $\theta$,~$\eta\in\R$.
Then a finite positive Borel measure $\mu$ is a vanishing $\theta$-Carleson measure if and only if
$\delta^\eta\mu$ is a vanishing $(\theta+\frac{\eta}{n+1})$-Carleson measure.
\end{lemma}

\begin{proof}
Assume $\mu$ is vanishing $\theta$-Carleson, set $\mu_\eta=\delta^\eta\mu$, and choose $r\in(0,1)$.
Then Lemmas~\ref{sette} and~\ref{sei} yield
\[
\begin{aligned}
\frac{\mu_\eta\bigl(B_D(z_0,r)\bigr)}{\nu\bigl(B_D(z_0,r)\bigr)^{\theta+\frac{\eta}{n+1}}}&=
\frac{1}{\nu\bigl(B_D(z_0,r)\bigr)^{\theta+\frac{\eta}{n+1}}}\int_{B_D(z_0,r)}\delta(\zeta)^\eta\,d\mu(\zeta)\\
&\preceq \frac{\delta(z_0)^\eta}{\nu\bigl(B_D(z_0,r)\bigr)^{\theta+\frac{\eta}{n+1}}} \mu\bigl(B_D(z_0,r)\bigr)\\
&\preceq \frac{\mu\bigl(B_D(z_0,r)\bigr)}{\nu\bigl(B_D(z_0,r)\bigr)^\theta}\to 0\;,
\end{aligned}
\]
as $z_0\to\de D$ because $\mu$ is vanishing $\theta$-Carleson, 
and so $\mu_\eta$ is vanishing $\left(\theta+\frac{\eta}{n+1}\right)$-Carleson. Since $\mu=(\mu_\eta)_{-\eta}$, the converse follows too.
\end{proof}

On the other hand,  uniformly discrete sequences yield vanishing $\theta$-Carleson measures only if they are finite:

\begin{theorem} 
\label{vandtre} 
Let $D\subset\subset\C^n$ be a bounded strongly pseudoconvex domain, considered as a metric space with the distance~$\rho_D=\tanh k_D$, and choose $1-\frac{1}{n+1}<\theta<2$. Let $\Gamma=\{z_j\}_{j\in\N}$ be a sequence
in~$D$. Then the following assertions are equivalent:
\begin{itemize}
\item[(i)] $\lim\limits_{z_0\to\de D}N(z_0,r,\Gamma)=0$ for all $r\in(0,1)$;
\item[(ii)] there exists $r_0>0$ such that $\lim\limits_{z_0\to\de D}N(z_0,r_0,\Gamma)=0$;
\item[(iii)] $\mu_\theta=\sum_j \delta(z_j)^{(n+1)\theta}\delta_{z_j}$ is a vanishing $\theta$-Carleson measure, where $\delta_{z_j}$ is the Dirac measure in~$z_j$;
\item[(iv)] $\Gamma$ is finite.
\end{itemize}
\end{theorem}

\begin{proof}
(i)$\Longrightarrow$(ii) is obvious.
\smallbreak

(ii)$\Longrightarrow$(iii) We have
\[
\begin{aligned}
\mu_\theta\bigl(B_D(z_0,r_0)\bigr)&=\sum_{z_j\in B_D(z_0,r_0)\cap\Gamma}\delta(z_j)^{(n+1)\theta}\preceq \delta(z_0)^{(n+1)\theta}N(z_0,r_0,\Gamma)\cr
&\preceq \nu\bigl(B_D(z_0,r_0)\bigr)^\theta
N(z_0,r_0,\Gamma)\;,
\end{aligned}
\]
where we used Lemmas~\ref{sei} and~\ref{sette}, and we are done.
\smallbreak

(iii)$\Longrightarrow$(i). Fix $r\in(0,1)$, and let $\delta_r>0$ be given by Lemma~\ref{piu}. 
Then if $\delta(z_0)<\delta_r$ using as usual Lemmas~\ref{sette} and~\ref{sei} we obtain
\[
\begin{aligned}
N(z_0,r,\Gamma)&\preceq\sum_{z\in B_D(z_0,r)\cap\Gamma}\delta(z_0)^{n+1}|k_{z_0}(z)|^2\cr
&\preceq \delta(z_0)^{(n+1)(1-\theta)}\sum_{z\in B_D(z_0,r)\cap\Gamma}\delta(z)^{(n+1)\theta}|k_{z_0}(z)|^2\cr
&=\| \delta(z_0)^{(n+1)(1-\theta)/2}k_{z_0}\|^2_{L^2(\mu_\theta)}\;.
\end{aligned}
\]
Then Theorem~\ref{BKp} and Lemma~\ref{BK} imply that $\{\delta(z_0)^{(n+1)(1-\theta)/2}k_{z_0}\}$ is norm-bounded in~$L^2(\mu_\theta)$ and converges to~$0$ uniformly on compact subsets as $z_0\to\de D$. Then $\| \delta(z_0)^{(n+1)(1-\theta)/2}k_{z_0}\|^2_{L^2(\mu_\theta)}\to 0$ as $z_0\to\de D$ by Theorem~\ref{Carleson} and Lemma~\ref{weakly}, and we are done.
\smallbreak
(iii)$\Longrightarrow$(i). Obvious.
\smallbreak
(ii)$+$(iii)$\Longrightarrow$(iv). By (ii) there is $\delta_0>0$ such that $N(z_0,r_0,\Gamma)=0$ if
$\delta(z_0)<\delta_0$; in particular, $\Gamma\cap\{z\in D\mid\delta(z)<\delta_0\}=\emptyset$, and thus
$\Gamma$ is contained in a relatively compact subset of~$D$. But the fact that $\mu_\theta$ is a finite measure implies that $\Gamma$ intersects any relatively compact subset of~$D$ in a
finite set, and we are done.
\end{proof}

\begin{remark}
The implication (ii)$\Longrightarrow$(iii) holds for any $\theta>0$. 
\end{remark}

\section{Toeplitz operators}

\begin{definition}
\label{Toeplitz}
Let $D\subset\subset\C^n$ be a bounded domain. The \emph{Toeplitz
operator}~$T_\mu$ associated to a finite positive Borel measure $\mu$ on~$D$ is defined by
\[
T_\mu f(z)=\int_D K(z,w) f(w)\, d\mu(w)\;.
\]
\end{definition}

%

Our aim in this section is to study mapping properties
of Toeplitz operators by means of Carleson properties of the measures. Our first main result shows that being $\theta$-Carleson implies
that the associated Toeplitz operator gives a gain in integrability if we use the correct weights:


\begin{theorem}
\label{Toeptre}
Let $D\subset\subset\C^n$ be a bounded strongly pseudoconvex domain. Given $1< p<+\infty$, let $p'$ be the conjugate exponent of $p$. Choose $\theta>1-\frac{1}{n+1}\min\left(1,\frac{1}{p-1}\right)$, and
let $\mu$ be a $\theta$-Carleson measure on~$D$. Then:
\begin{itemize}
\item[(i)]If $\theta<1$ and $p\le r< \frac{p'}{(n+1)(1-\theta)}$, then $T_\mu\colon A^p\bigl(D,(n+1)p(\theta-1+\frac{1}{r}-\frac{1}{p})\bigr)
\to A^r(D)$ continuously; 
\item[(ii)] if $1\le\theta< p'$ and $\frac{p'}{p'-\theta}<r<+\infty$, then $T_\mu\colon A^p\bigl(D,(n+1)p(\theta-1+\frac{1}{r}-\frac{1}{p})\bigr)
\to A^r(D)$ continuously; 
\item[(iii)] if $1\le\theta< p'$ and $p\le r\le\frac{p'}{p'-\theta}$, or 
if $p'\le\theta$ and $p\le r<+\infty$, then $T_\mu\colon A^p\bigl(D,(n+1)(\theta-1-\eps)\bigr)
\to A^r(D)$ continuously for all $\eps>0$; 
\end{itemize}
Furthermore, in all cases if $\mu$ is vanishing then $T_\mu$ is a compact operator between the given spaces.
\end{theorem}

\begin{proof}
Fix $1< p\le s\le r<\infty$, and denote by $s'$, respectively $r'$, the conjugate exponent of~$s$, respectively~$r$. Then 
\[
\begin{aligned}
|T_\mu f(\zeta)|&\le\int_D|K(\zeta,w)||f(w)|\,d\mu(w)\cr
&=\int_D|K(\zeta,w)|^{1/s}|f(w)||K(\zeta,w)|^{1/s'}\,d\mu(w)\cr
&\le \left[\int_D|K(\zeta,w)|^{p/s}|f(w)|^p\,d\mu(w)\right]^{1/p}\left[\int_D|K(\zeta,w)|^{p'/s'}\,d\mu(w)
\right]^{1/p'}\cr
&\preceq \left[\int_D|K(\zeta,w)|^{p/s}|f(w)|^p\,d\mu(w)\right]^{1/p}\left[\int_D|K(\zeta,w)|^{p'/s'}
\delta(w)^{(n+1)(\theta-1)}\,d\nu(w)
\right]^{1/p'}\;,
\end{aligned}
\]
using H\"older's inequality and the fact that $\mu$ is $\theta$-Carleson (notice that $p'/s'\ge 1$). Using 
Theorem~\ref{BKp} we then get
\begin{equation}
|T_\mu f(\zeta)|\preceq 
\begin{cases}
\displaystyle
\left[\int_D|K(\zeta,w)|^{p/s}|f(w)|^p\,d\mu(w)\right]^{1/p}
\delta(\zeta)^{(n+1)(\theta-p'/s')/p'}&\hbox{if $1-\frac{1}{n+1}<\theta<\frac{p'}{s'}$,}\\
\displaystyle
\left[\int_D|K(\zeta,w)|^{p/s}|f(w)|^p\,d\mu(w)\right]^{1/p}
\bigl|\log|\delta(\zeta)|\bigr|^{1/p'}&\hbox{if $\theta=\frac{p'}{s'}$,}\\
\displaystyle
\left[\int_D|K(\zeta,w)|^{p/s}|f(w)|^p\,d\mu(w)\right]^{1/p}&\hbox{if $\theta>\frac{p'}{s'}$.}
\end{cases}
\label{eq:cases}
\end{equation}
Let us first consider the case $1-\frac{1}{n+1}<\theta<\frac{p'}{s'}$. Using Minkowski's integral inequality (see \cite[6.19]{Fo}) we obtain
\[
\begin{aligned}
\|T_\mu f\|_r^p&\preceq
\left[\int_D\left[\int_D|K(\zeta,w)|^{p/s}|f(w)|^p\delta(\zeta)^{(n+1)(\theta-p'/s')(p/p')}\,d\mu(w)
\right]^{r/p}d\nu(\zeta)\right]^{p/r}\\
&\le\int_D|f(w)|^p\left[\int_D|K(\zeta,w)|^{r/s}\delta(\zeta)^{(n+1)(\theta-p'/s')(r/p')}\,d\nu(\zeta)
\right]^{p/r}d\mu(w)\;.
\end{aligned}
\]
Since we are assuming $\theta<\frac{p'}{s'}$ and $\frac{r}{s}\ge 1$, we automatically have
$(n+1)(\theta-p'/s')(r/p')<(n+1)\bigl((r/s)-1\bigr)$. Thus if
\begin{equation}
-1<(n+1)\left(\theta-\frac{p'}{s'}\right)\frac{r}{p'}
\label{eq:cond}
\end{equation}
we can again apply Theorem~\ref{BKp} obtaining
\begin{equation}
\|T_\mu f\|_r^p\preceq 
\int_D|f(w)|^p\delta(w)^{(n+1)p\left[\frac{1}{p'}(\theta-1)+\frac{1}{r}-\frac{1}{p}\right]} d\mu(w)\;.
\label{eq:vaninc}
\end{equation}
Now, Lemma~\ref{changeCarl} says that $\delta^{(n+1)p\left[\frac{1}{p'}(\theta-1)+\frac{1}{r}-\frac{1}{p}\right]}\mu$ is $\left(\theta+p\left[\frac{1}{p'}(\theta-1)+\frac{1}{r}-\frac{1}{p}\right]\right)$-Carleson;
therefore Theorem~\ref{carthetaCarluno} and Remark~\ref{rem:uno} yield
\begin{equation}
\|T_\mu f\|_r^p\preceq \int_D|f(w)|^p\delta(w)^{(n+1)p\left[\theta-1+\frac{1}{r}-\frac{1}{p}\right]}d\nu(w)\;,
\label{eq:inclu}
\end{equation}
that is $T_\mu$ maps $A^p\bigl(D,(n+1)p(\theta-1+\frac{1}{r}-\frac{1}{p})\bigr)$ into $A^r(D)$ continuously.  

We have obtained this assuming $1-\frac{1}{n+1}<\theta<\frac{p'}{s'}$ and (\ref{eq:cond}); now we would like to choose $s$ so that $r$ can be as large as possible. Assuming $\theta<\frac{p'}{s'}$, condition (\ref{eq:cond}) is equivalent to
\begin{equation}
r<\frac{p'}{(n+1)\bigl(\frac{p'}{s'}-\theta\bigr)}\;.
\label{eq:conddue}
\end{equation}
Since $\frac{p'}{s'}\ge 1$ always, if $\theta<1$ the right-hand side of (\ref{eq:conddue}) is largest 
for $\frac{p'}{s'}=1$, that is for $s=p$, and thus we get
\[
r<\frac{p'}{(n+1)(1-\theta)}\;.
\]
Since $r\ge p$, to ensure a not empty statement we have to require
\[
p<\frac{p'}{(n+1)(1-\theta)} \quad \Longleftrightarrow\quad 1-\frac{1}{(n+1)(p-1)}<\theta\;,
\]
and thus we have proved (i).

Assume now $1\le\theta<p'$ and $r>\frac{p'}{p'-\theta}$. Then 
$1\le\theta<\frac{p'}{r'}$, and we can find $p\le s\le r$ so that
\[
\frac{1}{p'}\le\frac{\theta}{p'}<\frac{1}{s'}<\frac{\theta}{p'}+\min\left\{\frac{1}{(n+1)r},\frac{1}{r'}-
\frac{\theta}{p'}\right\}\le\frac{1}{r'}\;;
\]
With such a choice we have $\theta<\frac{p'}{s'}$ and
(\ref{eq:cond}) is satisfied; so (\ref{eq:inclu}) holds, and we have proved~(ii).

If instead $1\le\theta<p'$ and $p< r\le \frac{p'}{p'-\theta}$, or $\theta'\ge p'$ and $r>p$,
for any $p\le s<r$ we have $\theta>\frac{p'}{s'}$, and so we must use the third line in (\ref{eq:cases}).
Relying again on Minkowski's integral inequality we get
\[
\begin{aligned}
\|T_\mu f\|_r^p&\preceq \left[\int_D\left[\int_D |K(\zeta,w)|^{p/s}|f(w)|^p d\mu(w)\right]^{r/p}d\nu(\zeta)\right]^{p/r}\\
&\le \int_D |f(w)|^p\left[\int_D|K(\zeta,w)|^{r/s}\,d\nu(\zeta)\right]^{p/r}d\mu(w)\;;
\end{aligned}
\]
then, since $s<r$, by Theorem~\ref{BKp} we get
\begin{equation}
\|T_\mu f\|_r^p\preceq \int_D |f(w)|^p\delta(w)^{-(n+1)\frac{p}{r}\left(\frac{r}{s}-1\right)}d\mu(w)\;.
\label{eq:cvanu}
\end{equation}
Applying Lemma~\ref{changeCarl}, Theorem~\ref{carthetaCarluno} and Remark~\ref{rem:uno}
we obtain
\[
\|T_\mu f\|_r^p\preceq \int_D |f(w)|^p \delta(w)^{(n+1)\left[\theta-1-p\left(\frac{1}{s}-\frac{1}{r}\right)\right]}d\nu(w)\;;
\]
so choosing $s$ close enough to~$r$ we see that $T_\mu$ maps $A^p\bigl(D,(n+1)(\theta-1-\eps)\bigr)$ continuously into $A^r(D)$ for all $\eps>0$.

To complete the proof of (iii) we have to deal with the case $\theta\ge 1$ and $p=r$ (and thus $p=s$ too). If $\theta=1$ we must apply the second line in (\ref{eq:cases}), obtaining
\begin{equation}
\begin{aligned}
\|T_\mu f\|^p_p&\preceq \int_D\left[\int_D|K(\zeta,w)||f(w)|^p\,d\mu(w)\right]\delta(\zeta)^{-\eta/p'}d\nu(\zeta)\\
&=\int_D|f(w)|^p\left[\int_D|K(\zeta,w)|\delta(\zeta)^{-\eta/p'}d\nu(\zeta)\right]d\mu(w)\\
&\preceq \int_D |f(w)|^p\delta(w)^{-\eta/p'}d\mu(w)
\end{aligned}
\label{eq:cvand}
\end{equation}
for all $0<\eta<p'$, and thus
\[
\|T_\mu f\|_p^p \preceq \int_D |f(w)|^p \delta(w)^{-\eps}d\nu(w)
\]
for all $0<\eps<1$, and (iii) is proved in this case too. Finally, if $\theta>1$ we must apply the
third line in (\ref{eq:cases}); so 
\begin{equation}
\begin{aligned}
\|T_\mu f\|^p_p&\preceq \int_D\left[\int_D|K(\zeta,w)||f(w)|^p\,d\mu(w)\right]d\nu(\zeta)\\
&=\int_D|f(w)|^p\left[\int_D|K(\zeta,w)|d\nu(\zeta)\right]d\mu(w)\\
&\preceq \int_D |f(w)|^p\delta(w)^{-(n+1)\eps}d\mu(w)
\end{aligned}
\label{eq:altra}
\end{equation}
for all $0<\eps$, and thus
\[
\|T_\mu f\|_p^p \preceq \int_D |f(w)|^p \delta(w)^{(n+1)(\theta-1-\eps)}d\nu(w)
\]
for all $0<\eps$, and (iii) again holds.

Finally, assume that $\mu$  is vanishing. In cases (i) and (ii), equation (\ref{eq:vaninc})
implies that
$T_\mu$ maps $L^p(\delta^\alpha\mu)$ into $L^r(D)$ continuously, where
\[
\alpha
=
(n+1)p\left(\frac{1}{p'}(\theta-1)+\frac{1}{r}-\frac{1}{p}\right)\;.
\]
Now, by Lemma~\ref{vCshift} we know that $\delta^\alpha\mu$ is vanishing $(\theta +\alpha/(n+1))$-Carleson; therefore
Theorem~\ref{Carleson} and Remark~\ref{rem:C} imply that the inclusion $\iota_\alpha\colon A^p(D,(n+1)(\theta +\alpha/(n+1) -1))\equiv A^p\bigl(D,(n+1)p(\theta-1+\frac{1}{r}-\frac{1}{p})\bigr) \hookrightarrow L^p(\delta^\alpha\mu)$ is compact. So $T_\mu\colon A^p\bigl(D,(n+1)p(\theta-1+\frac{1}{r}-\frac{1}{p})\bigr) \to A^r(D)$ is obtained as the composition of a bounded operator with a compact operator,
and hence is compact. A similar argument works in case (iii), replacing
(\ref{eq:vaninc}) by (\ref{eq:cvanu}), (\ref{eq:cvand}) or (\ref{eq:altra}) according to the situation. 
\end{proof}

\begin{remark}
Notice that if $1\le\theta<p'$ then
\[
p\left(\theta-1+\frac{1}{r}-\frac{1}{p}\right)\ge \theta-1\quad\Longleftrightarrow\quad
r\le \frac{p'}{p'-\theta}\;;
\]
therefore the domain of definition of $T_\mu$ is always the smallest between
$A^p\bigl(D,(n+1)p(\theta-1+\frac{1}{r}-\frac{1}{p})\bigr)$ and $A^p\bigl(D,(n+1)(\theta-1-\eps)\bigr)$.
\end{remark}

\begin{remark}
\label{rem:Lp}
In the previous proof we used the fact that the argument of~$T_\mu$ is a holomorphic function
only to go from \eqref{eq:vaninc} to \eqref{eq:inclu}, because we used there that $\mu$ is a Carleson measure for a suitable weighted Bergman spaces $A^p(D,\beta)$. But if $\mu=\delta^\eta\nu$ 
with $\eta=(n+1)(\theta-1)$ then the step from \eqref{eq:vaninc} to \eqref{eq:inclu} works for all
$f\in L^p\bigl(\delta^{(n+1)p(\theta-1+\frac{1}{r}-\frac{1}{p})}\nu\bigr)$, and thus we have shown that $T_{\delta^\eta\nu}$ maps continuously $L^p\bigl(\delta^{(n+1)p(\frac{\eta}{n+1}+\frac{1}{r}-\frac{1}{p})}\nu \bigr)$ into~$L^r(D)$.
%
\end{remark}

Choosing the parameters so that the weight is positive we obtain the following:

\begin{corollary}
\label{corToeptre}
Let $D\subset\subset\C^n$ be a bounded strongly pseudoconvex domain. Given $1< p<+\infty$, let $p'$ be the conjugate exponent of $p$. Choose $\theta>1-\frac{1}{n+1}\min\left(1,\frac{1}{p-1}\right)$, and
let $\mu$ be a $\theta$-Carleson measure on~$D$. Then:
\begin{itemize}
\item[(i)] if $\frac{p'}{p'-\theta}< r<+\infty$ and $1+\frac{1}{p}-\frac{1}{r}\le\theta< p'$, then $T_\mu\colon A^p(D)
\to A^r(D)$ continuously; 
\item[(ii)] if $1<\theta<p'$ and $p\le r\le\frac{p'}{p'-\theta}$, or $p'\le\theta$ and $p\le r$, then $T_\mu\colon A^p(D)
\to A^r(D)$ continuously. 
\end{itemize}
Furthermore, in both cases if $\mu$ is vanishing then $T_\mu$ is a compact operator between the given spaces.
\end{corollary}

\begin{proof}
It follows immediately from Theorem~\ref{Toeptre}, because Lemma~\ref{compwn} 
implies $A^p(D)\hookrightarrow A^p(D,\beta)$ continuously for all $\beta\ge0$. 
\end{proof}

\begin{remark}
We might also consider a \emph{weighted Toeplitz operator}
\[
T^\beta_\mu f(z)=\int_D K(z,w)f(w)\delta(w)^\beta\,d\mu(w)\;.
\]
Since if $\mu$ is $\theta$-Carleson we know that $\delta^\beta\mu$ is $\left(\theta+\frac{\beta}{n+1}\right)$-Carleson, as a consequence of Theorem~\ref{Toeptre} we obtain that
if $1<p< r$ then $T^{-(n+1)(\theta-1+\frac{1}{r} -\frac{1}{p})}_\mu\colon A^p(D)\to A^r(D)$ continuously.
Indeed, if $1<p<r$ then putting $\theta'=1+\frac{1}{p}-\frac{1}{r}$ we have $1< \theta'<p'$ and $\frac{p'}{p'-\theta'}< r$ always; and clearly $\theta+\frac{\beta}{n+1}=\theta'$ if
and only if $\beta=-(n+1)(\theta-1+\frac{1}{r} -\frac{1}{p})$. 
\end{remark}

%
%

We now consider the case $p>1$ and $r=+\infty$:

\begin{theorem}
\label{Toepinftydue}
Let $D\subset\subset\C^n$ be a bounded strongly pseudoconvex domain, let $1<p<+\infty$ and choose
$\theta\ge 1$. Let $\mu$ be a $\theta$-Carleson measure on~$D$. Then
\begin{itemize}
\item[(i)] if $1\le\theta<p'$, then $T_\mu\colon A^p\bigl(D,(n+1)p(\theta-1-\frac{1}{p}-\eps)\bigr)
\to A^\infty(D)$ continuously for all $\eps>0$, where $p'$ is the conjugate exponent of $p$;
\item[(ii)] if $\theta\ge p'$, then 
$T_\mu\colon A^p\bigl(D,(n+1)(\theta-1-\eps)\bigr)
\to A^\infty(D)$ continuously for all $\eps>0$.
\end{itemize}
If moreover $\mu$ is vanishing, then $T_\mu$ is compact.
\end{theorem}

\begin{proof}
When $1<\theta<p'$, set $\eta=p-(p-1)\theta$, so that $0<\eta< 1$. Given $\eps>0$
so that $\eta+p\eps<1$ set $s=p/(\eta+p\eps)$. Then 
$p< s<p/\eta$ and $\theta=\frac{p-\eta}{p-1}>\frac{p'}{s'}$, where $s'$ is the conjugate exponent of~$s$; so 
(\ref{eq:cases}) and Theorem~\ref{BKp}.(iv) yield
\[
\|T_\mu f\|_\infty\preceq \left[\int_D|f(w)|^p \delta(w)^{-(n+1)p/s}\,d\mu(w)\right]^{1/p}
\preceq \|f\|_{p,(n+1)(\theta-1-\frac{p}{s})}\;.
\]
Since $\theta-1-\frac{p}{s}=p(\theta-1-\frac{1}{p}-\eps)$, we are done in this case.

When $\theta\ge p'$ we can argue in a similar way choosing $s=p/\eps$, since $\theta\ge p'>\frac{p'}{s'}$.
Furthermore, when $\theta=1$ we use a similar argument based on the second line of 
(\ref{eq:cases}) with $s=p$. Finally, the statement for $\mu$ vanishing follows as 
in the proof of Theorem~\ref{Toeptre}.
\end{proof}

The case $p=1$ and $r<+\infty$ is completely analogous to Theorem~\ref{Toeptre}, 
noticing that $\theta+\frac{1}{r}-2=\theta-1+\frac{1}{r}-1$:

\begin{theorem}
\label{Toepqua}
Let $D\subset\subset\C^n$ be a bounded strongly pseudoconvex domain. Given $0<\theta$,
let $\mu$ be a $\theta$-Carleson measure on~$D$. Then:
\begin{itemize}
\item[(i)]$T_\mu\colon A^1\bigl(D,(n+1)(\theta+\frac{1}{r}-2)\bigr)
\to A^r(D)$ continuously for all $1<r<+\infty$.
\item[(ii)] $T_\mu\colon A^1\bigl(D,(n+1)(\theta-1-\eps)\bigr)
\to A^1(D)$ continuously for all small $\eps>0$.
\end{itemize}
Furthermore, in both cases if $\mu$ is vanishing then $T_\mu$ is compact.
\end{theorem}

\begin{proof}
Using again Minkowski's integral inequality (or plain Fubini's theorem when $r=1$) and
Theorem~\ref{BKp} we obtain
\[
\begin{aligned}
\|T_\mu f\|_r &\le \left[\int_D\left(\int_D |K(\zeta,w)||f(w)|\,d\mu(w)\right)^r d\nu(\zeta)\right]^{1/r}\cr
&\le
\int_D|f(w)|\left[\int_D|K(\zeta,w
)|^r\,d\nu(\zeta)\right]^{1/r}\,d\mu(w)\cr
&\preceq
\begin{cases}
\displaystyle
\int_D|f(w)|\delta(w)^{-\frac{n+1}{r'}}\,d\mu(w)&\hbox{if $r>1$,}\\
\noalign{\smallskip}
\displaystyle
\int_D|f(w)|\delta(w)^{-(n+1)\eps}\,d\mu(w)&\hbox{if $r=1$,}
\end{cases}\cr
\end{aligned}
\]
where $r'$ is the conjugated exponent of $r$ when $r>1$, and $\eps>0$ is arbitrary when~$r=1$.
Recalling Lemma~\ref{changeCarl}, Theorem~\ref{carthetaCarluno} and Remark~\ref{rem:uno}
we obtain
\[
\|T_\mu f\|_r\preceq 
\begin{cases}
\displaystyle
\int_D|f(w)|\delta(w)^{(n+1)\left(\theta+\frac{1}{r}-2\right)}\,d\nu(w)&\hbox{if $r>1$,}\\
\noalign{\smallskip}
\displaystyle
\int_D|f(w)|\delta(w)^{(n+1)(\theta-1-\eps)}\,d\nu(w)&\hbox{if $r=1$,}
\end{cases}
\]
and we have proved (i) and (ii).

Assume finally that $\mu$  is vanishing. The previous computation implies that when $r>1$ the Toeplitz operator
$T_\mu$ maps $L^1(\delta^{-(n+1)/r'}\mu)$ into $L^r(D)$ continuously.
Now, by Lemma~\ref{vCshift} we know that $\delta^{-(n+1)/r'}\mu$ is vanishing $\left(\theta-\frac{1}{r'}\right)$-Carleson; therefore
Theorem~\ref{Carleson} and Remark~\ref{rem:C} imply that the inclusion $\iota\colon A^1\bigl(D,(n+1)(\theta+\frac{1}{r} -2)\bigr)\hookrightarrow L^1(\delta^{-(n+1)/r'}\mu)$ is compact. So $T_\mu\colon A^1\bigl(D,(n+1)(\theta+\frac{1}{r} -2)\bigr) \to A^r(D)$ is obtained as the composition of a bounded operator with a compact operator,
and hence is compact. A similar argument works for $r=1$.
\end{proof}

We also have a statement for $p=1$ and $r=+\infty$:

\begin{theorem}
\label{Toepinfty}
Let $D\subset\subset\C^n$ be a bounded strongly pseudoconvex domain, and choose
$0<\theta$. Let $\mu$ be a $\theta$-Carleson measure on~$D$. Then
$T_\mu\colon A^1\bigl(D,(n+1)(\theta-2)\bigr)
\to A^\infty(D)$ continuously. If moreover $\mu$ is vanishing then $T_\mu$ is compact.
\end{theorem}

\begin{proof}
Using Theorem~\ref{BKp}.(iv), arguing as usual we obtain
\[
\begin{aligned}
\|T_\mu f\|_\infty&\le\sup_{z\in D}\int_D|K(z,w)||f(w)|\,d\mu(w)\\
&\preceq \int_D |f(w)|\delta(w)^{-(n+1)}\,d\mu(w)
\preceq \int_D |f(w)|\delta(w)^{(n+1)(\theta-2)}\,d\nu(w)\;,
\end{aligned}
\]
as claimed. Furthermore, when $\mu$ is vanishing the usual argument works, 
and we are done.
\end{proof}

\begin{remark}
In particular, if $\mu$ is Carleson, then $T_\mu$ maps $A^1\bigl(D, -(n+1)\bigr)$ into $A^\infty(D)$ continuously.
\end{remark}


We would like now to investigate the converse implications, using mapping properties
of the Toeplitz operator to infer Carleson properties of the measure. 

A piece of notation: if $f$, $g\colon D\to\C$ are such that $f\bar{g}\in L^1(D)$ we shall write
\[
\langle f,g\rangle=\int_D f(z)\overline{g(z)}\,d\nu(z)\;.
\]
Then the main result linking Toeplitz operators and Carleson properties is the following basic fact:

\begin{proposition}
\label{converse}
Let $\mu$ be a finite positive Borel measure on a bounded domain
$D\subset\subset\C^n$. Then
\[
B\mu(z_0)=\langle T_\mu k_{z_0},k_{z_0}\rangle\;.
\]
\end{proposition}

\begin{proof}
Indeed using the reproducing property of the Bergman kernel we have
\[
\begin{aligned}
B\mu (z_0)
&=\int_D \frac{|K(w,z_0)|^2}{K(z_0, z_0)}  \, d\mu(w)\\
&=\int_D \frac{K(w,z_0)}{K(z_0, z_0)} K(z_0,w) \, d\mu(w)\\
&=\int_D \frac{K(w,z_0)}{K(z_0, z_0)}\left( \int_D K(x,w) K(z_0,x)\, d\nu(x) \right)\, d\mu(w)\\
&=\int_D \left( \int_D \frac{K(w,z_0)}{\sqrt{K(z_0, z_0)}} K(x,w) \, d\mu(w)\right) \frac{\overline{K(x,z_0)}}{\sqrt{K(z_0, z_0)}}\, d\nu(x)\\
&=\int_D \left( \int_D  K(x,w) k_{z_0}(w) \, d\mu(w)\right) \overline{k_{z_0}(x)}\, d\nu(x)\\
&=\left\langle T_\mu  k_{z_0},  k_{z_0}\right\rangle.
\end{aligned}
\]
\end{proof}

Let us begin with the case $1<p\le r<+\infty$:

\begin{theorem}
\label{Toepinvgen}
Let $D\subset\subset\C^n$ be a bounded strongly pseudoconvex domain. 
Let $\mu$ be a finite positive Borel measure on $D$.
Given $1< p<+\infty$, assume that 
$T_\mu\colon A^p\bigl(D,(n+1)\beta\bigr)
\to A^r(D)$ continuously for some $p\le r<+\infty$.
Then 
\begin{itemize}
\item[(i)] if  $-\frac{1}{n+1}<\beta<p-1$, then $\mu$ is $\left(1+\frac{\beta}{p}+\frac{1}{p}-\frac{1}{r}\right)$-Carleson;
\item[(ii)] if $\beta=p-1$, then $\mu$ is $\left(2-\frac{1}{r}-\eps\right)$-Carleson
for all $\eps>0$;
\item[(iii)] if $\beta>p-1$, then $\mu$ is $\left(2-\frac{1}{r}\right)$-Carleson.
\end{itemize}
Furthermore, in cases \emph{(i)} and \emph{(ii)} if $T_\mu$ is compact then $\mu$ is vanishing.
\end{theorem}

\begin{proof}
Denoting by $r'$ the conjugate exponent of $r$,
Proposition~\ref{converse}, H\"older's inequality and the assumption yield
\[
B\mu(z_0)=\langle T_\mu k_{z_0},k_{z_0}\rangle\le \|T_\mu k_{z_0}\|_r\|k_{z_0}\|_{r'}
\preceq \|k_{z_0}\|_{p,(n+1)\beta}\|k_{z_0}\|_{r'}\;.
\]
We can now use Theorem~\ref{BKp}. In case (i) we have
\[
B\mu(z_0)\preceq\delta(z_0)^{(n+1)\left[\frac{1}{2}+\frac{\beta}{p}-\frac{1}{p'}+\frac{1}{2}-\frac{1}{r}\right]}=\delta(z_0)^{(n+1)\left[\frac{\beta}{p}+\frac{1}{p}-\frac{1}{r}\right]}\;,
\]
and the assertion follows from Theorem~\ref{carthetaCarldue}.

Analogously, in case (ii) we have
\[
B\mu(z_0)\preceq\delta(z_0)^{(n+1)\left[\frac{1}{2}-\eps+\frac{1}{2}-\frac{1}{r}\right]}=\delta(z_0)^{(n+1)\left[1-\eps-\frac{1}{r}\right]}\;,
\]
for all $\eps>0$, and again the assertion follows from Theorem~\ref{carthetaCarldue}. Case (iii) is identical.

Finally, assume that $T_\mu$ is compact, and set $\theta=1+\frac{\beta}{p}+\frac{1}{p}-\frac{1}{r}$
in case (i), $\theta=2-\frac{1}{r}-\eps$ in case (ii), and $\theta=2-\frac{1}{r}$ in case (iii).
Then Proposition~\ref{converse} yields
\[
\begin{aligned}
\delta(z_0)^{(n+1)(1-\theta)}B\mu(z_0)&\le \delta(z_0)^{(n+1)(1-\theta)}\|T_\mu k_{z_0}\|_r\|k_{z_0}\|_{r'}\cr
&\preceq 
\begin{cases}
\delta(z_0)^{(n+1)\left(\frac{1}{2}-\frac{1}{p}-\frac{\beta}{p}\right)}\|T_\mu k_{z_0}\|_r
&\hbox{if $-\frac{1}{n+1}<\beta<p-1$,}\\
\delta(z_0)^{(n+1)\left(\eps-\frac{1}{2}\right)}\|T_\mu k_{z_0}\|_r&\hbox{if $\beta=p-1$,}\\
\delta(z_0)^{-\frac{n+1}{2}}\|T_\mu k_{z_0}\|_r&\hbox{if $\beta>p-1$.}
\end{cases}
\end{aligned}
\]
If we denote by $\eta$ the exponent of~$\delta(z_0)$ in cases (i) and (ii), 
$\{\delta(z_0)^\eta k_{z_0}\}_{z_0\in D}$ is 
bounded in $A^p\bigl(D,(n+1)\beta\bigr)$ by Theorem~\ref{BKp}, and converges to~0
uniformly on compact subsets as $z_0\to\de D$ by Lemma~\ref{BK}; therefore the
compactness of~$T_\mu$ together with Lemma~\ref{ucsweak} and Proposition~\ref{th:gencompact}
yield $\delta(z_0)^\eta\|T_\mu k_{z_0}\|_r\to 0$
as $z_0\to\de D$, and the assertion follows from Theorem~\ref{carvanthetaCarldue}
and Remark~\ref{rem:Cb}.
\end{proof}

\begin{remark}
Since $\delta(z_0)^{-(n+1)/2}k_{z_0}$ does not converge to~0 uniformly on compact subsets as $z_0\to\de D$ but it is merely uniformly bounded, in case (iii) we cannot conclude that $\mu$ is vanishing.
\end{remark}

\begin{remark}
Case (i) for $\beta=0$ and $p=r$ shows that 
if $T_\mu\colon A^p(D)\to A^p(D)$ is continuous (respectively, compact) then $\mu$ is (respectively, vanishing) Carleson. 
\end{remark}

We have a similar statement for $p=1$ too, but the proof that if $T_\mu$ is compact then $\mu$ is vanishing requires a few preliminary lemmas:

\begin{lemma}
\label{th:prelemuno}
Let $D\subset\subset\C^n$ be a bounded strongly pseudoconvex domain, choose
$\eps\ge0$ and $\max\left\{1-\frac{1}{n+1},1-\eps\right\}<\theta$, and let $\mu$ be a $\theta$-Carleson measure. Then:
\begin{itemize}
\item[(i)] if $\eps\ge 1$, then $\{\delta(z_0)^{(n+1)(\eps-\frac{1}{2})}k_{z_0}\}_{z_0\in D}$ is norm-bounded in $L^s(\mu)$ for all $s\ge1$;
\item[(ii)] if $0\le\eps<1$ and $\theta>1$, then $\left\|\delta(z_0)^{(n+1)(\eps-\frac{1}{2})}k_{z_0}\right\|_{L^s(\mu)}\preceq\delta(z_0)^{(n+1)\eps}$ for all $1\le s<\theta$;
\item[(iii)] if $0\le\eps<1$ and $\max\left\{1-\frac{1}{n+1},1-\eps\right\}<\theta\le 1$, then 
$\{\delta(z_0)^{(n+1)(\eps-\frac{1}{2})}k_{z_0}\}_{z_0\in D}$ is norm-bounded in $L^s(\mu)$
for all $1\le s\le\frac{\theta}{1-\eps}$. 
\end{itemize}
\end{lemma}

\begin{proof}
If $\eps\ge 1$ then $\{\delta(z_0)^{(n+1)(\eps-\frac{1}{2})}k_{z_0}\}_{z_0\in D}$ is uniformly 
bounded by Theorem~\ref{BKp}.(iv), and (i) follows.

Assume then $0\le\eps<1$. Then using as usual Theorem~\ref{carthetaCarluno}, Remark~\ref{rem:uno} and Theorem~\ref{BKp} for any $s\ge 1$ we obtain
\[
\begin{aligned}
\|\delta(z_0)^{(n+1)(\eps-\frac{1}{2})}k_{z_0}\|_{L^s(\mu)}&=
\delta(z_0)^{(n+1)(\eps-\frac{1}{2})}\left[\int_D|k_{z_0}(\zeta)|^s\,d\mu(\zeta)\right]^{1/s}\cr
&\preceq \delta(z_0)^{(n+1)(\eps-\frac{1}{2})}\left[\int_D|k_{z_0}(\zeta)|^s\delta(\zeta)^{(n+1)(\theta-1)}\,d\nu(\zeta)\right]^{1/s}\cr
&\preceq
\begin{cases}
\delta(z_0)^{(n+1)(\eps+\frac{\theta}{s}-1)}&\hbox{if $\theta<s$,}\cr
\delta(z_0)^{(n+1)(\eps-\eta)}&\hbox{for any $\eta>0$ if $\theta=s$,}\cr
\delta(z_0)^{(n+1)\eps}&\hbox{if $\theta>s$,}\cr
\end{cases} 
\end{aligned}
\]
and (ii) and (iii) follow.
\end{proof}

\begin{lemma}
\label{th:prelemdue}
Let $D\subset\subset\C^n$ be a bounded strongly pseudoconvex domain, and $\mu$ a positive finite Borel measure on~$D$. Assume that $\{f_k\}$ is a sequence converging to~$0$ uniformly on compact subsets and norm-bounded in $L^s(\mu)$ for some $1<s\le+\infty$. Then $\{T_\mu f_k\}$ converges to~$0$ uniformly on compact subsets.
\end{lemma}

\begin{proof}
Fix $D_0\subset\subset D$. As we already noticed, \cite[Theorem~2]{K} implies that $|K(z,w)|\le C$ for all $z\in D_0$ and $w\in D$; therefore
\[
|T_\mu f_k(z)|\le\int_D|K(z,w)||f_k(w)|\,d\mu(w)\le C\int_D|f_k(w)|\,d\mu(w)
\]
for all $z\in D_0$. So it suffices to show that $\int_D|f_k(w)|\,d\mu(w)\to 0$ as $k\to+\infty$ knowing
that $f_k\to 0$ uniformly on compact subsets and $\|f_k\|_{L^s(\mu)}\le M$ for some $1<s\le+\infty$. If $s=+\infty$ the assertion follows from the dominated convergence theorem; assume then
$1<s<+\infty$, and let $s'$ be its conjugate exponent. Given $\eps>0$, choose $\eta>0$
so that $\mu(D_\eta)<(\eps/2M)^{s'}$, where $D_\eta=\{w\in D\mid \delta(w)<\eta\}$. 
Choose now $k_0$ so that 
\[
\sup_{w\in D\setminus D_\eta}|f_k(w)|\le\frac{\eps}{2\mu(D\setminus D_\eta)}
\]
for all $k\ge k_0$. Then
\[
\begin{aligned}
\int_D|f_k(w)|d\mu(w)&\le\int_{D\setminus D_\eta}|f_k(w)|d\mu(w)+\int_{D_\eta}|f_k(w)|d\mu(w)\\
&\le \frac{\eps}{2}+\|f_k\|_{L^s(\mu)}\mu(D_\eta)^{1/s'}\le\eps
\end{aligned}
\]
for all $k\ge k_0$, and we are done.
\end{proof}

\begin{corollary}
\label{th:prelemtre}
Let $D\subset\subset\C^n$ be a bounded strongly pseudoconvex domain, choose
$\eps>0$ and $\theta>\max\left\{1-\frac{1}{n+1},1-\eps\right\}$, and let $\mu$ be a $\theta$-Carleson measure. Then $\delta(z_0)^{(n+1)(\eps-\frac{1}{2})}T_\mu k_{z_0}\to 0$ uniformly on compact subsets as $z_0\to\de D$.
\end{corollary}

\begin{proof}
It follows immediately from Lemmas~\ref{BK}, \ref{th:prelemuno} and~\ref{th:prelemdue}. 
\end{proof}

We can now deal with the case $p=1\le r<+\infty$:

\begin{theorem}
\label{ToepinvLuno}
Let $D\subset\subset\C^n$ be a bounded strongly pseudoconvex domain. 
Let $\mu$ be a finite positive Borel measure on $D$.
Assume that 
$T_\mu\colon A^1\bigl(D,(n+1)\beta\bigr)
\to A^r(D)$ continuously for some $1\le r<+\infty$.
Then 
\begin{itemize}
\item[(i)] if  $-\frac{1}{n+1}<\beta<0$, then $\mu$ is $\left(2+\beta-\frac{1}{r}\right)$-Carleson;
\item[(ii)] if $\beta=0$, then $\mu$ is $\left(2-\frac{1}{r}-\eps\right)$-Carleson
for all $\eps>0$;
\item[(iii)] if $\beta>0$, then $\mu$ is $\left(2-\frac{1}{r}\right)$-Carleson.
\end{itemize}
Furthermore, in cases \emph{(i)} and \emph{(ii)} if $T_\mu$ is compact and $r>1$ then $\mu$ is vanishing.
\end{theorem}

\begin{proof}
The first part of the proof goes exactly as for Theorem~\ref{Toepinvgen}:
denoting by $r'$ the conjugate exponent of $r$,
Proposition~\ref{converse}, the H\"older inequality and the assumption yield
\[
B\mu(z_0)=\langle T_\mu k_{z_0},k_{z_0}\rangle\le \|T_\mu k_{z_0}\|_r\|k_{z_0}\|_{r'}
\preceq \|k_{z_0}\|_{1,(n+1)\beta}\|k_{z_0}\|_{r'}\;.
\]
We can now use Theorem~\ref{BKp}. In case (i) we have
\[
B\mu(z_0)\preceq\delta(z_0)^{(n+1)\left[1+\beta-\frac{1}{r}\right]}\;,
\]
and the assertion follows from Theorem~\ref{carthetaCarldue}.

Analogously, in case (ii) we have
\[
B\mu(z_0)\preceq\delta(z_0)^{(n+1)\left[1-\eps-\frac{1}{r}\right]}\;,
\]
for all $\eps>0$, and again the assertion follows from Theorem~\ref{carthetaCarldue}. Case (iii) is identical.

Finally, assume that $T_\mu$ is compact; in this case the argument is slightly different because we cannot apply Lemma~\ref{ucsweak} and Proposition~\ref{th:gencompact}.
Anyway, set $\theta=2+\beta-\frac{1}{r}$
in case (i), $\theta=2-\frac{1}{r}-\eps$ in case (ii), and $\theta=2-\frac{1}{r}$ in case (iii).
Then Proposition~\ref{converse} yields
\[
\begin{aligned}
\delta(z_0)^{(n+1)(1-\theta)}B\mu(z_0)&\le \delta(z_0)^{(n+1)(1-\theta)}\|T_\mu k_{z_0}\|_r\|k_{z_0}\|_{r'}\cr
&\preceq 
\begin{cases}
\delta(z_0)^{(n+1)\left(-\beta-\frac{1}{2}\right)}\|T_\mu k_{z_0}\|_r
&\hbox{if $-\frac{1}{n+1}<\beta<0$,}\\
\delta(z_0)^{(n+1)\left(\eps-\frac{1}{2}\right)}\|T_\mu k_{z_0}\|_r&\hbox{if $\beta=0$,}\\
\delta(z_0)^{-\frac{n+1}{2}}\|T_\mu k_{z_0}\|_r&\hbox{if $\beta>0$.}
\end{cases}
\end{aligned}
\]
If we denote by $\eta$ the exponent of~$\delta(z_0)$ in cases (i) and (ii), $\{\delta(z_0)^\eta k_{z_0}\}_{z_0\in D}$ is 
norm-bounded in $A^1\bigl(D,(n+1)\beta\bigr)$ by Theorem~\ref{BKp}, and 
$\delta(z_0)^\eta T_\mu k_{z_0}$ converges to~0
uniformly on compact subsets as $z_0\to\de D$ by Corollary~\ref{th:prelemtre} (that we can use because $r>1$); we claim 
that $\delta(z_0)^\eta \|T_\mu k_{z_0}\|_r\to 0$. If not, we can find a sequence $z_j\to\de D$
and $\delta>0$ such that $\delta(z_j)^\eta\|T_\mu k_{z_j}\|_r\ge\delta$ for all $j\in\N$. Now,
since $\{\delta(z_0)^\eta k_{z_0}\}_{z_0\in D}$ is norm-bounded and $T_\mu$ is compact,
up to a subsequence we can assume that $\delta(z_j)^\eta T_\mu k_{z_j}\to h\in A^r(D)$ strongly.
But we know that $\delta(z_j)^\eta T_\mu k_{z_j}\to 0$ uniformly on compact subsets; therefore
$h\equiv 0$ and thus $\delta(z_j)^\eta\|T_\mu k_{z_j}\|_r\to 0$, contradiction.

So $\delta(z_0)^\eta \|T_\mu k_{z_0}\|_r\to 0$ as $z_0\to\de D$, and the assertion 
follows from Theorem~\ref{carvanthetaCarldue}
and Remark~\ref{rem:Cb}.
\end{proof}

We finally have a statement for $1\le p<+\infty$ and $r=+\infty$ too:

\begin{theorem}
\label{Toepinvpinfty}
Let $D\subset\subset\C^n$ be a bounded strongly pseudoconvex domain. 
Let $\mu$ be a finite positive Borel measure on $D$.
Given $1\le  p<+\infty$, assume that 
$T_\mu\colon A^p\bigl(D,(n+1)\beta\bigr)
\to A^\infty(D)$ continuously.
Then 
\begin{itemize}
\item[(i)] if  $-\frac{1}{n+1}<\beta\le p-1$, then $\mu$ is $\left(1+\frac{\beta}{p}+\frac{1}{p}-\eps\right)$-Carleson for all $\eps>0$;
\item[(ii)] if $\beta>p-1$, then $\mu$ is $\left(2-\eps\right)$-Carleson for all $\eps>0$.
\end{itemize}
Furthermore, in both cases $\mu$ is vanishing.
\end{theorem}

\begin{proof}
Proposition~\ref{converse}, H\"older's inequality and the assumption yield
\[
B\mu(z_0)=\langle T_\mu k_{z_0},k_{z_0}\rangle\le \|T_\mu k_{z_0}\|_\infty\|k_{z_0}\|_1
\preceq \|k_{z_0}\|_{p,(n+1)\beta}\|k_{z_0}\|_1\;.
\]
We can now use Theorem~\ref{BKp}. In case (i) with $\beta<p-1$ we have
\[
B\mu(z_0)\preceq\delta(z_0)^{(n+1)\left[\frac{1}{2}+\frac{\beta}{p}-\frac{1}{p'}+\frac{1}{2}-\eps\right]}=\delta(z_0)^{(n+1)\left[\frac{\beta}{p}+\frac{1}{p}-\eps\right]}\;,
\]
and the assertion follows from Theorem~\ref{carthetaCarldue}.

Analogously, in case (i) with $\beta=p-1$ we have
\[
B\mu(z_0)\preceq\delta(z_0)^{(n+1)\left[\frac{1}{2}-\eps+\frac{1}{2}-\eps\right]}=\delta(z_0)^{(n+1)\left[1-2\eps\right]}\;,
\]
for all $\eps>0$, and again the assertion follows from Theorem~\ref{carthetaCarldue}. Case (ii) is identical, and the final assertion follows from Remark~\ref{rem:carvancar}.
\end{proof}

\begin{remark}
\label{rem:inftyinfty}
A similar argument shows that if $T_\mu\colon A^\infty\bigl(D,(n+1)\beta\bigr)\to A^\infty(D)$ 
continuously then
\begin{itemize}
\item[(i)] if $0\le\beta<1$ then $\mu$ is vanishing $(1+\beta-\eps)$-Carleson for all~$\eps>0$, and
\item[(ii)] if $\beta\ge 1$ then $\mu$ is vanishing $(2-\eps)$-Carleson for all~$\eps>0$.
\end{itemize}
\end{remark}

We finally summarize our results giving a few ``if and only if" statements. We begin with 
some general though technical results:

\begin{corollary}
\label{Toepinvdue}
Let $D\subset\subset\C^n$ be a bounded strongly pseudoconvex domain. Given $1< p<+\infty$, 
choose $1-\frac{1}{(n+1)p}<\theta< 1$. Then a finite positive Borel
measure~$\mu$ on $D$ is $\theta$-Carleson (respectively, vanishing $\theta$-Carleson) if and only if 
$T_\mu\colon A^p\bigl(D,(n+1)p(\theta-1+\frac{1}{r}-\frac{1}{p})\bigr)
\to A^r(D)$ continuously (respectively, is compact) for some (and hence all) $p\le r< \min\left(\frac{p(n+1)}{(n+1)p(1-\theta)+n},\frac{p'}{(n+1)(1-\theta)}\right)$, where $p'$ is the conjugate exponent of $p$.
\end{corollary}

\begin{proof}
One direction follows from Theorem~\ref{Toeptre}.(i), while the converse follows from
Theorem~\ref{Toepinvgen}.(i) applied to $\beta=p(\theta-1+\frac{1}{r} -\frac{1}{p})$; notice that
the assumption on $r$ ensures that $-\frac{1}{n+1}<\beta<p-1$, and the assumption
on~$\theta$ ensures that $p<\frac{p(n+1)}{(n+1)p(1-\theta)+n}$.
\end{proof}

\begin{corollary}
\label{Toepinvuno}
Let $D\subset\subset\C^n$ be a bounded strongly pseudoconvex domain, and choose $1< p
<+\infty$.
Then a positive finite Borel measure $\mu$ on~$D$ is $\theta$-Carleson (respectively, vanishing $\theta$-Carleson) for all $\theta<1$ if and only if 
$T_\mu\colon A^p\bigl(D,(n+1)p(\frac{1}{r}-\frac{1}{p}-\eps)\bigr)
\to A^r(D)$ continuously (respectively, is compact) for some (and hence all) $p\le r<p\left(1+\frac{1}{n}\right)$ and all $\eps>0$.
\end{corollary}

\begin{proof}
It follows from the previous corollary.
\end{proof}

\begin{corollary}
\label{Toepinvtre}
Let $D\subset\subset\C^n$ be a bounded strongly pseudoconvex domain, and choose 
$1-\frac{1}{n+1}<\theta\le 1$. Then a positive finite Borel measure $\mu$ on~$D$ is $\theta$-Carleson (respectively, vanishing $\theta$-Carleson) if and only if 
$T_\mu\colon A^1\bigl(D,(n+1)(\theta-1+\frac{1}{r}-1)\bigr)
\to A^r(D)$ continuously (respectively, is compact) for some (and hence all) $1<r<\frac{n+1}{(n+1)(2-\theta)-1}$. 
In particular,
$\mu$ is Carleson (respectively, vanishing Carleson) if and only if $T_\mu\colon A^1\bigl(D,(n+1)(\frac{1}{r}-1)\bigr)
\to A^r(D)$ continuously (respectively, is compact) for some (and hence all) $1<r<1+\frac{1}{n}$.
\end{corollary}

\begin{proof}
One direction is Theorem~\ref{Toepqua}, while the converse follows from Theorem~\ref{ToepinvLuno} applied with $\beta=\theta-1+\frac{1}{r}-1$; notice that the assumption on~$r$
ensures that $-\frac{1}{n+1}< \beta \le 1$, and the assumption on~$\theta$ ensures that
$1<\frac{n+1}{(n+1)(2-\theta)-1}$.
\end{proof}

We obtain more expressive corollaries if we strive for clarity instead of generality:

\begin{corollary}
\label{th:coruno}
Let $D\subset\subset\C^n$ be a bounded strongly pseudoconvex domain. 
Let $\mu$ be a finite positive Borel measure on $D$, and take $1< p< r<+\infty$. Then the following statements are equivalent:
\begin{itemize}
\item[(i)] $T_\mu\colon A^p(D)\to A^r(D)$ continuously (respectively, compactly);
\item[(ii)] $\mu$ is (respectively, vanishing) $\left(1+\frac{1}{p}-\frac{1}{r}\right)$-Carleson.
\end{itemize}
\end{corollary}

\begin{proof}
(i)$\Longrightarrow$(ii) follows immediately from Theorem~\ref{Toepinvgen}, while (ii)$\Longrightarrow$(i)
follows from Theorem~\ref{Toeptre} noticing that if $\theta=1+\frac{1}{p}-\frac{1}{r}$ then $1<p<r$ implies
$1<\theta<p'$ and $\frac{p'}{p'-\theta}<r$.
\end{proof}

\begin{remark}
The implication (i)$\Longrightarrow$(ii) holds for $p=r$ too; the best result we have for the reverse implication when $p=r$ is Corollary~\ref{th:cortre} below.
\end{remark}

As recalled in the introduction, in \cite{CMc} Cu\v ckovi\'c and McNeal studied special Toeplitz operators of the form
\[
T_{\delta^\eta} f(z)=\int_D K(z,w)f(w)\delta(w)^\eta\,d\nu(w)\;.
\]
In our context, $T_{\delta^\eta}=T_{\delta^\eta\nu}$; since $\delta^\eta\nu$ is $\left(1+\frac{\eta}{n+1}\right)$-Carleson by Lemma~\ref{changeCarl}, we can recover the main Theorem~1.2 
of~\cite{CMc} as a consequence of our Corollary~\ref{th:coruno}; in particular, it follows 
that the gain in the exponents proved in \cite{CMc} is sharp. 

\begin{corollary}[\textbf{\cite[Theorem 1.2]{CMc}}]
\label{cor:CMc}
Let $D\subset\subset\C^n$ be a bounded strongly pseudoconvex domain, and let $\eta\ge 0$.
\begin{itemize}
\item[(a)] If $0\le\eta< n+1$, then:
\begin{itemize}
\item[(i)] if $1< p< \infty$ and $\frac{n+1}{n+1-\eta}<\frac{p}{p-1}$, then
$T_{\delta^\eta}\colon L^p(D)\to L^{p+G}(D)$ continuously, where $G=p^2\big/\left(\frac{n+1}{\eta}-p\right)$;
\item[(ii)] if $1< p< \infty$ and $\frac{n+1}{n+1-\eta}\ge\frac{p}{p-1}$, then $T_{\delta^\eta}\colon L^p(D)\to L^r(D)$ continuously for all $p\le r<\infty$.
\end{itemize}
\item[(b)] If $\eta\ge n+1$, then $T_{\delta^\eta}\colon L^1(D)\to L^\infty(D)$ continuously.
\end{itemize}
\end{corollary}

\begin{proof}
First of all, notice that $0\le\eta<n+1$ and $\frac{n+1}{n+1-\eta}<\frac{p}{p-1}$ are equivalent to requiring that $0\le\frac{\eta}{n+1}<\frac{1}{p}$, and thus $1+\frac{1}{p}-\theta>0$,
where $\theta=1+\frac{\eta}{n+1}$. Then
\[
1+\frac{\eta}{n+1}=1+\frac{1}{p}-\frac{1}{r} \quad\Longleftrightarrow\quad r=p+G\;,
\]
and thus (a).(i) follows immediately from Corollary~\ref{th:coruno} and Remark~\ref{rem:Lp}.

Analogously, if $\frac{n+1}{n+1-\eta}\ge\frac{p}{p-1}$ we have $\frac{\eta}{n+1}\ge\frac{1}{p}$ 
and thus, setting again $\theta=1+\frac{\eta}{n+1}$, we have $\theta>1+\frac{1}{p}-\frac{1}{r}$ for all $r\ge p$; therefore Corollary~\ref{th:coruno} and Remark~\ref{rem:Lp} again imply that $T_{\delta^\eta}$ maps $L^p(D)$ into $L^r(D)$ continuously for all $p\le r<\infty$, that is (a).(ii). 

Finally, (b) is a trivial consequence of Theorem~\ref{BKp}.(iv). 
\end{proof}

\begin{corollary}
\label{th:cordue}
Let $D\subset\subset\C^n$ be a bounded strongly pseudoconvex domain. 
Let $\mu$ be a finite positive Borel measure on $D$, and take $1\le p< r<p\left(1+\frac{1}{n}\right)$. Then the following statements are equivalent:
\begin{itemize}
\item[(i)] $T_\mu\colon A^p\bigl(D,(n+1)p(\frac{1}{r}-\frac{1}{p})\bigr)\to A^r(D)$ continuously (respectively, compactly);
\item[(ii)] $\mu$ is (respectively, vanishing) Carleson.
\end{itemize}
\end{corollary}

\begin{proof} 
When $p>1$,
(i)$\Longrightarrow$(ii) follows immediately from Theorem~\ref{Toepinvgen} with $-\frac{1}{n+1}<\beta=\frac{p}{r}-1<0$, while (ii)$\Longrightarrow$(i)
follows from Theorem~\ref{Toeptre} noticing that $\frac{p'}{p'-1}=p$.

When $p=1$, one instead uses Theorems~\ref{Toepqua} and~\ref{ToepinvLuno}.
\end{proof}

\begin{corollary}
\label{th:cortre}
Let $D\subset\subset\C^n$ be a bounded strongly pseudoconvex domain. 
Let $\mu$ be a finite positive Borel measure on $D$, and take $1\le p<+\infty$. Then the following statements are equivalent:
\begin{itemize}
\item[(i)] $T_\mu\colon A^p\bigl(D,-(n+1)\eps\bigr)\to A^p(D)$ continuously (respectively, compactly) 
for all $\eps>0$;
\item[(ii)] $\mu$ is (respectively, vanishing) $\theta$-Carleson for all $\theta<1$.
\end{itemize}
\end{corollary}

\begin{proof}
It follows immediately from Theorems~\ref{Toepinvgen} and \ref{Toeptre} when $p>1$,
and from Theorems~\ref{Toepqua} and~\ref{ToepinvLuno} when $p=1$.
\end{proof}

%
\begin{corollary}
\label{th:cortreb}
Let $D\subset\subset\C^n$ be a bounded strongly pseudoconvex domain. 
Let $\mu$ be a finite positive Borel measure on $D$, and take $1< r<+\infty$. Then the following statements are equivalent:
\begin{itemize}
\item[(i)] $T_\mu\colon A^1\bigl(D,-(n+1)\eps\bigr)\to A^r(D)$ continuously (respectively, compactly) 
for all $\eps>0$;
\item[(ii)] $\mu$ is (respectively, vanishing) $\theta$-Carleson for all $\theta<2-\frac{1}{r}$.
\end{itemize}
\end{corollary}

\begin{proof}
It follows immediately from Theorems~\ref{Toepqua} and \ref{ToepinvLuno}.
\end{proof}

\begin{corollary}
\label{th:corqua}
Let $D\subset\subset\C^n$ be a bounded strongly pseudoconvex domain. 
Let $\mu$ be a finite positive Borel measure on $D$, and take $1\le p<+\infty$. Then the following statements are equivalent:
\begin{itemize}
\item[(i)] $T_\mu\colon A^p\bigl(D,-(n+1)\eps\bigr)\to A^\infty(D)$ continuously (respectively, compactly) 
for all $\eps>0$;
\item[(ii)] $\mu$ is (respectively, vanishing) $\theta$-Carleson for all $\theta<1+\frac{1}{p}$.
\end{itemize}
\end{corollary}

\begin{proof}
It follows immediately from Theorems~\ref{Toepinfty}, \ref{Toepinftydue} and \ref{Toepinvpinfty}.
\end{proof}

\begin{remark}
\label{rem:fine}
The techniques we introduced can clearly be used to study mapping properties of Toeplitz
operators having unweighted Bergman spaces as domain and weighted Bergman spaces
as codomain; we leave the details to the interested reader.
\end{remark}

\end{document}